\documentclass[12pt]{article}
\usepackage[active]{srcltx}

\usepackage{amsthm,amsfonts,amscd,amsmath,amssymb,color,array,verbatim,psfrag,graphicx,caption2}

\def\cbar{\overline{\C}}

\def\be{\beta}

\def\G{\Gamma}

\def\De{\Delta}
\def\de{\delta}
\def\tl{\tilde}

\def\R{\mbox{$\mathbb R$}}

\def\C{\mbox{$\mathbb C$}}
\def\T{\mbox{$\mathbb T$}}

\def\D{\mbox{$\mathbb D$}}

\def\P{\bf P}

\def\N{\mbox{$\mathbb N$}}

\def\TTT{{\mathcal T}}
\def\AAA{{\mathcal A}}
\def\BBB{{\mathcal B}}

\def\DDD{{\mathcal D}}

\def\KKK{{\bf K}}

\def\LLL{{\mathcal L}}

\newtheorem{newthm}{Theorem}
\newtheorem{claim}{Claim}
\newtheorem{theorem}{Theorem}[section]
\newtheorem{lemma}[theorem]{Lemma}
\newtheorem{proposition}[theorem]{Proposition}
\newtheorem{corollary}[theorem]{Corollary}
\newtheorem{definition}{Definition}
\newtheorem*{main}{Main Proposition}
\newcommand{\REFEQN}[1] { \begin{equation}\label{#1} }
\newcommand{\ENDEQN}{\end{equation}}
\newcommand{\REFTHM}[1] { \begin{theorem}\label{#1} }
\newcommand{\ENDTHM}{\end{theorem}}
\newcommand{\REFNTH}[1] { \begin{newthm}\label{#1} }
\newcommand{\ENDNTH}{\end{newthm}}
\newcommand{\REFPROP}[1]{\begin{proposition}\label{#1} }
\newcommand{\ENDPROP}{\end{proposition} }
\newcommand{\REFLEM}[1]{\begin{lemma}\label{#1} }
\newcommand{\ENDLEM}{\end{lemma} }
\newcommand{\REFCOR}[1]{\begin{corollary}\label{#1} }
\newcommand{\ENDCOR}{\end{corollary} }
\newcommand{\REFDEFTHM}[1] { \begin{defthm}\label{#1} }
\newcommand{\ENDDEFTHM}{\end{defthm}}

\def\smm{\backslash }
\def\ds{\displaystyle }

\begin{document}
\begin{center}{\large Combinatorial rigidity of multicritical maps} \vspace{0.2cm}  \\ Wenjuan
Peng
 \& Lei Tan \\ \today
\end{center}

\abstract We combine the KSS nest constructed by Kozlovski, Shen and
van Strien,
 and the analytic method proposed by Avila, Kahn, Lyubich and Shen to prove the
combinatorial rigidity of multicritical maps.

\vspace{4mm}

\section{Introduction}
Rigidity is one of the fundamental and remarkable phenomena in
holomorphic dynamics. The general rigidity problem can be posed as
\vspace{2mm}

\noindent{\bf Rigidity problem} \cite{L}. Any two combinatorially
equivalent rational maps are quasi-conformally equivalent. Except
for the Latt\`{e}s examples, the quasi-conformal deformations come
from the dynamics of the Fatou set.

\vspace{2mm}

In the quadratic polynomial case, the rigidity problem is equivalent
to the famous hyperbolic conjecture. The MLC conjecture asserting
that the Mandelbrot set is locally connected is stronger than the
hyperbolic conjecture (cf. \cite{DH}). In 1990, Yoccoz \cite{Hu}
proved MLC for all parameter values which are at most finitely
renormalizable. Lyubich \cite{L} proved MLC for infinitely
renormalizable quadratic polynomials of bounded type. In \cite{kss},
Kozlovski, Shen and van Strien gave a proof of the rigidity for real
polynomials with all critical points real. In \cite{four}, Avila,
Kahn, Lyubich and Shen proved that any unicritical polynomial
$f_c:z\mapsto z^d+c$ which is at most finitely renormalizable and
has only  repelling periodic points is combinatorially rigid, which
implies that the connectedness locus (the Multibrot set) is locally
connected at the corresponding parameter values.  The rigidity
problem for rational maps with Cantor Julia sets is totally solved
(cf. \cite{yz}, \cite{Z}). In \cite{Z}, Zhai took advantage of a
length-area method introduced by Kozlovski, Shen and van Strien (cf.
\cite{kss}) to prove the quasi-conformal rigidity for rational maps
with Cantor Julia sets. Kozlovski and van Strien proved that
topologically conjugate non-renormalizable polynomials are
quasi-conformally conjugate (cf. \cite{ks}).

In the following, we list some other cases in which the rigidity
problem is researched
(see also \cite{Z}).\\
(i) Robust infinitely renormalizable quadratic polynomials \cite{Mc1}.\\
(ii) Summable rational maps with small exponents \cite{GS}.\\
(iii) Holomorphic Collet-Eckmann repellers \cite{PR}.\\
(iv) Uniformly weakly hyperbolic rational maps \cite{Ha}.

In \cite{PT}, we discussed the combinatorial rigidity of unicritical
maps. In the present work, we give a proof of the combinatorial
rigidity of multicritical maps (see the definition in section 2).

We will exploit the powerful combinatorial tool called "puzzle" and
a sophisticated choice of puzzle pieces called the KSS nest
constructed in \cite{kss}.  To get the quasi-conformal conjugation,
we adapt the analytic method in \cite{four}, especially their Lemma
3.2.

This article is organized as follows. In section 2, we introduce the
definition of the multicritical maps and present our main results,
Theorems 2.1, 2.2, 2.3. In section 3, we apply the well-known
Spreading Principle to prove Theorem 2.1. In section 4, we resort to
the quasi-conformal surgery to prove Theorem 2.2. Proof of Theorem
2.3 (a) is given in section 5. We reduce Theorem 2.3 (b) to Main
Proposition in section 6. The proof of Main Proposition is presented
in section 7. In subsection 7.1, we reduce Main Proposition to
Proposition 7.3. The proof of Proposition 7.3 is given in subsection
7.2. The main result in \cite{cp} is to construct a holomorphic
model for a multiply-connected fixed (super)attracting Fatou
component of a rational map. In the appendix, we will apply Theorem
2.3 to give another proof of the quasi-conformal rigidity part of
their result.

\section{Statement}

\begin{center}\fbox{ $\begin{array}{c}
\text{{\bf The Set up}.\quad ${\bf V}=\sqcup_{i\in I}V_i$  is the disjoint
union of finitely many}\\ \text{ Jordan domains in the complex plane $\mathbb{C}$  with disjoint and quasi-circle boundaries,}\\
\text{ ${\bf U}$ is compactly contained in ${\bf V}$,} \\
 \text{and  is the union of finitely many
Jordan domains with disjoint closures;}\\
 \text{ $f:{\bf U}\to {\bf V}$ is a proper holomorphic map with all critical points} \\
\text{  contained in $\KKK_f:=\{z\in {\bf U}\mid f^n(z)\in {\bf U}\ \forall n\}$, such that}\\
\text{each connected component of ${\bf V}$}\\
\text{contains at most one connected component of $\KKK_f$ having
critical points.}\end{array}$}
\end{center}

Denote by $\mathrm{Crit}(f)$ the set of critical points of $f$ and
by ${\bf P}:=\bigcup_{n\ge 1}\bigcup_{c\in \mathrm{Crit}(f)}
\{f^n(c)\}$ the postcritical set of $f$.

Let $\mathrm{int}\KKK_f$ denote the interior of $\KKK_f$. For
$x\in\KKK_f$, denote by $\KKK_f(x)$ the component\footnote{In this
article, for simplicity, by a `component', we mean a `connected
component'.} of $\KKK_f$ containing the point $x$. We call a
component of $\KKK_f$ a \textit{critical} component if it contains a
critical point. The map $f$ maps each component of $\KKK_f$ (resp.
$\mathrm{int}\KKK_f$) onto a component of $\KKK_f$ (resp.
$\mathrm{int}\KKK_f$).  A component $K$ of $\KKK_f$ (resp.
$\mathrm{int}\KKK_f$) is called \emph{periodic} if $f^{p}(K)=K$ for
some $p\geq 1$, \emph{preperiodic} if $f^{n}(K)$ is periodic for
some $n\geq 0$, and \emph{wandering} otherwise, that is $f^i(K)\cap
f^j(K)=\emptyset$ for all $i\neq j\ge 0$.

Two maps in the set-up $(f:{\bf U}\to {\bf V}) ,(\tilde f:\tilde
{\bf U}\to \tilde {\bf V})$ are said to be {\em c-equivalent}
(combinatorially equivalent), if there is a pair of orientation
preserving homeomorphisms $h_0, h_1:{\bf V}\to \tilde{\bf V} $ such
that

\begin{equation*}\label{C-class}\left\{\begin{array}{l}h_1({\bf U})=\tilde{\bf U}\ \text{and}\ h_1(\P)=\tilde \P\\
  h_1 \text{ is isotopic to $h_0$ rel }\partial{\bf V}\cup \P\\
h_0\circ f\circ h_1^{-1}|_{\tilde {\bf U}}=\tilde f\\
h_1|_{\overline{\bf V}\smm {\bf U}}\text{ is $C_0$-qc (an
abbreviation of quasi-conformal) for some $C_0\ge 1$,}
\end{array}\right. \end{equation*}
$$\text{ in particular}\ \begin{array}{rcl}
{\bf V}\supset {\bf U} & \overset{h_1}{\longrightarrow} & \tilde{\bf U}\subset \tilde{\bf V} \vspace{0.09cm} \\
                                    f\downarrow && \downarrow \tilde f \vspace{0.09cm}\\
                    {\bf V} & \underset{h_0}{\longrightarrow} & \tilde{\bf V}
                    \end{array}\ \text{commutes}.$$
This definition is to be compared with the notion of combinatorial
equivalence introduced by McMullen in [Mc2]. Notice that this
definition is slightly different from the definitions of
combinatorial equivalence in \cite{four} and \cite{ks}, since we
define it without using external rays.

We say that $f$ and $\tl f$ are {\em qc-conjugate off $\KKK_f$} if
there is a qc map $H:{\bf V}\smm \KKK_f\to \tl {\bf V}\smm \KKK_{\tl
f}$ so that $H\circ f=\tl f\circ H$ on ${\bf U}\smm \KKK_f$,
$$\text{ i.e.}\ \begin{array}{lcl}
 {\bf U}\smm \KKK_f & \overset{H}{\longrightarrow} & \tilde{\bf U}\smm \KKK_{\tl f} \vspace{0.09cm} \\
                                    f\downarrow && \downarrow \tilde f \vspace{0.09cm}\\
                    {\bf V}\smm \KKK_f & \underset{H}{\longrightarrow} & \tilde{\bf V}\smm \KKK_{\tl f}
                    \end{array}\ \text{commutes}.$$

We say that $f$ and $\tl f$ are {\em qc-conjugate off
$\mathrm{int}\KKK_f$} if there is a qc map $\tl H:{\bf V}\to \tl
{\bf V}$
            so that $\tl H\circ f=\tl f\circ \tl H$ on ${\bf U}\smm \mathrm{int}\KKK_f$,
            $$\text{ i.e.}\ \begin{array}{lcl}
 {\bf U}\smm \mathrm{int}\KKK_f & \overset{\tl H}{\longrightarrow} & \tilde{\bf U}\smm \mathrm{int}\KKK_{\tl f} \vspace{0.09cm} \\
                                    f\downarrow && \downarrow \tilde f \vspace{0.09cm}\\
                    {\bf V}\smm \mathrm{int}\KKK_f & \underset{\tl H}{\longrightarrow}
                    & \tilde{\bf V}\smm \mathrm{int}\KKK_{\tl f}
                    \end{array}\ \text{commutes}.$$

 We say that $f$ and $\tl f$ are {\em qc-conjugate} if there is a qc
 map $H':{\bf V}\to \tl {\bf V}$ so that $H'\circ f=\tl f\circ H'$ on ${\bf U}$,  $$\text{ i.e.}\ \begin{array}{rcl}
{\bf V}\supset {\bf U} & \overset{H'}{\longrightarrow} & \tilde{\bf U}\subset \tilde{\bf V} \vspace{0.09cm} \\
                                    f\downarrow && \downarrow \tilde f \vspace{0.09cm}\\
                    {\bf V} & \underset{H'}{\longrightarrow} & \tilde{\bf V}
                    \end{array}\ \text{commutes}.$$

\begin{theorem}{\label{assumption}}
Let $f,\tl f$ be two maps in the set-up. Suppose that $f$ and $\tl
f$ are qc-conjugate off $\KKK_f$ by  a qc map $H$. Assume that the
following property $\mathrm{(\ast)}$ holds.

For every critical component $\KKK_f(c)$ of $\KKK_f$, $c\in
\mathrm{Crit}(f)$, and every integer $n\ge 1$, there exists a puzzle
piece $Q_n(c)$ containing $c$  such that:

(i) For every critical component $\KKK_f(c)$, the pieces
$\{Q_n(c)\}_{n\ge 1}$ form a nested sequence with $\bigcap_{n}
Q_n(c)=\KKK_f(c)$ (the depth of $Q_n(c)$ may not equal to $n$).

 (ii) For each $n\ge 1$, the union
$\bigcup_{c\in \mathrm{Crit}(f)}Q_n(c)$ is a nice set.

(iii) There is a constant $\tl C$, such that for each pair
$(n,\KKK_f(c))$ with $n\ge 1$ and $\KKK_f(c)$ a critical component ,
the map $H|_{\partial Q_n(c)}$ admits a $\tl C$-qc extension inside
$Q_n(c)$.

Then the map $H$ extends to a qc map from $\bf V$ onto $\tl {\bf V}$
which is a conjugacy off $\mathrm{int}\KKK_f$.
\end{theorem}

See Definition 1 (1) and (2) in the next section for the definitions
of a puzzle piece, the depth of it and a nice set.

\begin{theorem}{\label{renormalization}}
Let $f$ be a map in the set-up. Then $\mathrm{int}\KKK_f$ contains
no wandering components.
\end{theorem}

\begin{theorem}{\label{rigidity}}
Let $f,\tl f$ be two maps in the set-up. Then the following
statements hold.

(a) If $f$ and $\tl f$ are c-equivalent, then they are qc-conjugate
off $\KKK_f$.

(b) Suppose that $H:{\bf V}\smm \KKK_f\to \tl {\bf V}\smm \KKK_{\tl
f}$ is a qc-conjugacy off $\KKK_f$. Assume that for every critical
component $\KKK_f(c)$, $c\in \mathrm{Crit}(f)$, satisfying that
$f^l(\KKK_f(c))$ is a critical periodic component of $\KKK_f$ for
some $l\ge 1$ (including the case of $\KKK_f(c)$ periodic), there
are a constant $M_c$ and an integer $N_c\ge 0$ such that for each
$n\ge N_c$, the map $H|_{\partial P_n(c)}$ admits an $M_c$-qc
extension inside $P_n(c)$, where $P_n(c)$ is a puzzle piece of depth
$n$ containing $c$. Then the map $H$ extends to a qc-conjugacy off
$\mathrm{int}\KKK_f$. Furthermore, if for every preperiodic
component $K$ of $\KKK_f$ with non-empty interior, the map
$H|_{\partial K}$ extends to a qc-conjugacy inside $K$, then $f$ and
$\tl f$ are qc-conjugate by an extension of $H$.
\end{theorem}

\section{Proof of Theorem \ref{assumption}}
Suppose that $f$ and $\tl f$ are qc-conjugate off $\KKK_f$ by a
  $C_0$-qc map $H$. Starting from the property $(\mathrm{\ast})$, we will prove that $H$ admits a
 qc extension across $\KKK_f$ which
 is a conjugacy off $\mathrm{int}\KKK_f$.

\begin{definition}
(1) For every $n\ge 0$, we call each component of $f^{-n}({\bf V})$
a \textbf{puzzle piece} of \textbf{depth} $n$ for $f$. Similarly, we
call each component of $\tilde{f}^{-n}({\bf \tilde V})$ a puzzle
piece of depth $n$ for $\tilde{f}$. Denote by $\mathrm{depth}(P)$
the depth of a puzzle piece $P$.

We list below three basic properties of the puzzle pieces.

(a) Every puzzle piece is a quasi-disk and there are finitely many
puzzle pieces of the same depth.

(b) Given two puzzle pieces $P$ and $Q$ with
$\mathrm{depth}(P)>\mathrm{depth}(Q)$, either $P\subset\subset Q$ or
$\overline{P}\cap\overline{Q}=\emptyset$.

(c) For $x\in\KKK_f$, for every $n\ge 0$, there is a unique puzzle
piece of depth $n$ containing $x$. Denote the piece by $P_n(x)$.
Then $P_{n+1}(x)\subset\subset P_n(x)$ and $\cap_{n\ge 0}P_n(x)$ is
exactly the component of $\KKK_f$ containing $x$.

(2) Suppose that $X\subset {\bf V}$ is a finite union of puzzle
pieces (not necessarily of the same depth). We say that $X$ is
\textbf{nice} if for any $z\in
\partial X$ and any $n\ge 1$, $f^n(z)\notin X$ as long as $f^n(z)$
is defined, that is, for any component $P$ of $X$, for any $n\ge 1$,
$f^n(P)$ is not strictly contained in $X$. For example, if $X$ has a
unique component, obviously it is a nice set.

(3) Let $A$ be an open set and $z\in A$. Denote the component of $A$
containing $z$ by $\mathrm{Comp}_z(A)$.

Given an open set $X$ consisting of finitely many puzzle pieces, let
$$
D(X)=\{z\in{\bf V}\mid \exists k\ge 0, f^k(z)\in X\}=\cup_{k\ge
0}f^{-k}(X).
$$

For $z\in D(X)\smm X$, let $k(z)$ be the minimal positive integer
such that $f^{k(z)}(z)\in X$. Set
$$
\LLL_z(X):=\mathrm{Comp}_z(f^{-k(z)}(\mathrm{Comp}_{f^{k(z)}(z)}(X))).
$$
Obviously, $f^{k(z)}(\LLL_z(X))=\mathrm{Comp}_{f^{k(z)}(z)}(X)$.
\end{definition}

\begin{lemma}{\label{basic}}
Suppose that $X$ is a finite union of puzzle pieces. The following
statements hold.

(1) For any $z\in D(X)\smm X$, the sets $\LLL_z(X),
f(\LLL_z(X)),\cdots, f^{k(z)-1}(\LLL_z(X))$ are pairwise disjoint.

(2) Suppose that $X$ is nice and $z\in D(X)\smm X$. Then for all
$0\le i< k(z)$, we have $f^i(\LLL_z(X))\cap X=\emptyset$. In
particular, if $X\supset \mathrm{Crit}(f)$, then $\LLL_z(X)$ is
conformally mapped onto a component of $X$ by $f^{k(z)}$.
\end{lemma}

\begin{proof}
(1) Assume that there exist $0\le i<j<k(z)$ with $f^i(\LLL_z(X))\cap
f^j(\LLL_z(X))\ne \emptyset$. Then $f^i(\LLL_z(X))\subset\subset
f^j(\LLL_z(X))$ and
$$
f^{k(z)-j}(f^i(\LLL_z(X)))\subset\subset
f^{k(z)-j}(f^j(\LLL_z(X)))=f^{k(z)}(\LLL_z(X))=\mathrm{Comp}_{f^{k(z)}(z)}(X).
$$
So $f^{k(z)-j+i}(z)\in X$. But $0< k(z)-j+i<k(z)$. This is a
contradiction with the minimality of $k(z)$.

(2) Assume that there is some $0\le i_0<k(z)$ with
$f^{i_0}(\LLL_z(X))\cap X\ne\emptyset$. We can show
$f^{i_0}(\LLL_z(X))\cap X \subset\subset f^{i_0}(\LLL_z(X))$. In
fact, when $i_0\ne 0$, this is due to the minimality of $k(z)$; when
$i_0=0$, it is because $z\not\in X$. Let $P$ be a component of $X$
with $P\subset\subset f^{i_0}(\LLL_z(X))$. So
$f^{k(z)-i_0}(P)\subset\subset
f^{k(z)-i_0}(f^{i_0}(\LLL_z(X)))=\mathrm{Comp}_{f^{k(z)}(z)}(X)$. It
contradicts the condition that $X$ is nice.
\end{proof}

The corollary below follows directly from the above lemma.

\begin{corollary}{\label{nicecorollary}}
Suppose that $X$ is a finite union of puzzle pieces. The following
statements hold.

(i) For any $z\in D(X)\smm X$, the set $\{\LLL_z(X),
f(\LLL_z(X)),\cdots, f^{k(z)-1}(\LLL_z(X))\}$ meets every critical
point at most once and
$$
\deg(f^{k(z)}:\LLL_z(X)\to
\mathrm{Comp}_{f^{k(z)}(z)}(X))\le(\max_{c\in
\mathrm{Crit}(f)}\deg_{c}(f))^{\#\mathrm{Crit}(f)}
$$

(ii) Suppose that $X$ is nice and $z\in D(X)\smm X$. Then
$\LLL_w(X)=\LLL_z(X)$ for all $w\in\LLL_z(X)$ and
$\LLL_{w'}(X)\cap\LLL_z(X)=\emptyset$ for all $w'\not\in\LLL_z(X)$.

(iii) Suppose that $X$ is nice and $z\in D(X)\smm X$. Then
$f^i(\LLL_z(X))=\LLL_{f^i(z)}(X)$ for all $0< i< k(z)$.
\end{corollary}

Let $K$ be a critical component of $\KKK_f$ and $c_1,c_2,\cdots,
c_l$ be all the critical points on $K$. Then
$P_n(c_1)=P_n(c_2)=\cdots=P_n(c_l)$ and
$$
\deg(f|_{P_n(c_1)})=(\deg_{c_1}(f)-1)+\cdots+(\deg_{c_l}(f)-1)+1
$$ for all $n\ge 0$.
We can view $K$ as a component containing one critical point of
degree $\deg(f|_{P_n(c_1)})$. Hence in the following until the end
of this article, we assume that each $\bf V$-component contains at
most one critical point.

Now we will combine the property $(\ast)$ and the Spreading
Principle to prove Theorem \ref{assumption}.

\begin{proof}[Proof of Theorem \ref{assumption}]
First fix $n\ge 1$. We shall repeat the proof of the Spreading
Principle in \cite{kss} to get a qc map $H_n$ from $\bf V$ onto
$\bf\tilde V$.

Set $W_n:=\bigcup_{c\in \mathrm{Crit}(f)}Q_n(c)$. Then by Lemma
\ref{basic} (2), each component of $D(W_n)$ is mapped conformally
onto a component of $W_n$ by some iterate of $f$.

For every puzzle piece $P$, we can choose an arbitrary qc map
$\phi_P:P\to\tilde{P}$ with $\phi_P|_{\partial P}=H|_{\partial P}$
since $H$ is a qc map from a neighborhood of $\partial P$ to a
neighborhood of $\partial \tl P$ and $\partial P,\partial \tl P$ are
quasi-circles (see e.g. \cite{ct}, Lemma C.1). Note that by the
definition of $W_n$, there are finitely many critical puzzle pieces
not contained in $W_n$. So we can take $C'_n$ to be an upper bound
for the maximal dilatation of all the qc maps $\phi_P$, where $P$
runs over all puzzle pieces of depth 0 and all critical puzzle
pieces not contained in $W_n$.

Given a puzzle piece $P$, let $0\le k\le \mathrm{depth}(P)$ be the
minimal nonnegative integer such that $f^{k}(P)$ is a critical
puzzle piece or has depth 0. Set $\tau(P)= f^{k}(P)$. Then $f^k:P\to
\tau(P)$ is a conformal map and so is
$\tilde{f}^k:\tilde{P}\to\tau(\tilde{P})$, where $\tl P$  is the
puzzle piece bounded by $H(\partial P)$ for $\tl f$ and $\tau(\tl
P)=\tl f^k(\tl P)$. Given a qc map $q:\tau(P)\to \tau(\tilde{P})$,
we can lift it through the maps $f^k$ and $\tl f^k$, that is, there
is a qc map $p:P\to \tilde{P}$ such that $\tilde{f}^k\circ p=q\circ
f^{k}$. Notice that the maps $p$ and $q$ have the same maximal
dilatation, and if $q|_{\partial \tau(P)}=H|_{\partial \tau(P)}$,
then $p|_{\partial P}=H|_{\partial P}$.

Let $Y_0={\bf V}$ denote the union of all the puzzle pieces of depth
0. Set $X_0=\emptyset$. For $j\ge 0$, we inductively define
$X_{j+1}$ to be the union of puzzle pieces of depth $j+1$ such that
each of these pieces is contained in $Y_j$ and is a component of
$D(W_n)$; set $Y_{j+1}:=(Y_j\cap f^{-(j+1)}({\bf V}))\smm X_{j+1}$.
We have the following relations: for any $j\ge 0$,
$$ Y_j=(Y_j\smm f^{-(j+1)}({\bf V}))\sqcup X_{j+1}\sqcup Y_{j+1}, \ \ Y_{j+1}\subset\subset  Y_j, \  \ X_{j'}\cap X_j=\emptyset \text{ for any } j'\ne j\ .$$
Given any component $Q$ of $Y_{j+1}$, we claim that $\tau(Q)$ is
either one of the finitely many critical puzzle pieces not contained
in $W_n$, or one of the finitely many puzzle pieces of depth $0$. In
fact, for such $Q$, either $Q\cap D(W_n)=\emptyset$ or $Q\cap
D(W_n)\ne\emptyset$. In the former case, since
$\mathrm{Crit}(f)\subset W_n\subset D(W_n)$, the component $Q$ is
mapped conformally onto a puzzle piece of depth 0 by
$f^{\mathrm{depth}(Q)}$. So $\tau(Q)$ is a puzzle piece of depth 0.
In the latter case, if $Q\cap D(W_n)\subset\subset D(W_n)$, then $Q$
is compactly contained in a component of $D(W_n)$, denoted by $Q'$,
and $Q'\subset\subset X_{j'}$ for some $j'<j+1$. But
$Q\subset\subset Y_j\subset\subset Y_{j-1}\subset\subset \cdots
\subset\subset Y_0$ and $X_j\cap Y_j=\emptyset,X_{j-1}\cap
Y_{j-1}=\emptyset,\cdots,X_0\cap Y_0=\emptyset$. This is a
contradiction. Hence $Q\cap D(W_n)\subset\subset Q$. If there is a
critical point $c\in Q\cap D(W_n)$, then the component of $W_n$
containing $c$ is compactly contained in $Q$ and $\tau(Q)=Q$.
Otherwise, the set $\tau(Q)$ must be a critical puzzle piece not
contained in $W_n$.

Define $H^{(0)}=\phi_P$ on each component $P$ of $Y_0$. For each
$j\ge 0$, assuming that $H^{(j)}$  is defined, we define $H^{(j+1)}$
as follows:
$$H^{(j+1)}=\left\{\begin{array}{ll}H^{(j)}& \text{on } {\bf V}\smm Y_{j} \\
 H& \text{on }  Y_j\smm
f^{-(j+1)}({\bf
V})\\
\text{the univalent pullback of $\phi$}
 & \text{on each component of }  X_{j+1}\\
\text{the univalent pullback of $\phi_{\tau(Q)}$} & \text{on each
component }  Q \text{ of } Y_{j+1},
                    \end{array}\right.$$
where the map $\phi$ is the qc-extension obtained by the assumption
$(\ast)$.

Set $C_n=\max\{C_0,C'_n,\tl C\}$. The  $\{H^{(j)}\}_{j \geq 0}$ is a
sequence of $C_n$-qc maps. Hence it is precompact in the uniform
topology.

By definition,  $H^{(j)}=H^{(j+1)}$ outside $Y_{j}$. Thus, the
sequence $\{H^{(j)}\}$ converges pointwise  outside $$\bigcap_j Y_j=
\{x \in \KKK_f\mid f^k(x) \notin W_n, \ k \geq 0\}.$$ This set is a
hyperbolic subset, on which $f$ is uniformly expanding, and hence
has zero Lebesgue measure, in particular no interior.  So any two
limit maps of the sequence $\{H^{(j)}\}_{j \geq 0}$ coincide on a
dense open set of ${\bf V}$, therefore coincides on ${\bf V}$ to a
unique limit map. Denote this map by $H_n$. It is $C_n$-qc.

By construction, $H_n$ coincides with $H$ on ${\bf V}\smm
((\bigsqcup_jX_j)\cup(\bigcap_j Y_j))$, is therefore $C_0$-qc there;
and is $\tl C$-qc on $\bigsqcup_jX_j$. It follows that the maximal
dilatation of $H_n$ is bounded by $\max\{C_0,\tl C\}$ except
possibly on the set $ \bigcap_j Y_j. $ But this set has zero
Lebesgue measure. It follows that the maximal dilatation of $H_n$ is
$\max\{C_0,\tl C\}$, which is independent of $n$.

The sequence $H_n:{\bf V}\to \tl {\bf V}$ has a subsequence
converging uniformly to a limit qc map $H':{\bf V}\to \tl {\bf V}$,
with $H'|_{{\bf V}\smm \KKK_f}=H$. Therefore $H'$ is a qc extension
of $H$. On the other hand, $H\circ f=\tl f\circ H$ on ${\bf U}\smm
\KKK_f$. So $H'\circ f=\tl f\circ H'$ holds on ${\bf U}\smm
\mathrm{int}\KKK_f$ by continuity. Therefore $H'$ is a qc-conjugacy
off $\mathrm{int}\KKK_f$. This ends the proof of Theorem
\ref{assumption}.
\end{proof}

\section{Proof of Theorem \ref{renormalization}}

In this section let $f: {\bf U}\to {\bf V}$ be a map in the set-up.
First we will extend $f$ to a global quasi-regular map from the
Riemann sphere onto itself (see the following two lemmas). Let $q\ge
1$ denote the number of components of $\bf V$. Enumerate the ${\bf
V}$-components by $V_1,V_2,\cdots, V_q$.

\REFLEM{Inner}  Let
${\bf W}$ be an open  round disk centered at $0$
with radius $>1$ containing $\overline{\bf V}$. The map
 $f:{\bf U}\to {\bf V}$ extends to a map $F$ on ${\bf V}$ so that
\\
--on each component $V_i$ of ${\bf V}$, the restriction $F|_{V_i}: V_i\to \bf W$ is a quasi-regular branched covering;
\\
--every component of ${\bf U}$ is a component of $F^{-1}({\bf V})$;
\\
--the restriction $F$ on $F^{-1}({\bf V})$ is holomorphic.\ENDLEM

\begin{proof}

{\bf Part I}. Fix any component $V_i$ of ${\bf V}$ such that
$V_i\cap \bf U\neq\emptyset$. We will extend $f|_{V_i\cap {\bf U}}$
to a map $F$ on $V_i$ with the required properties. It will be done
in three steps. Refer to Figure \ref{construction-1} for the
construction of $F$ on $V_i$.

\begin{figure}[htbp]\centering
\includegraphics[width=8cm]{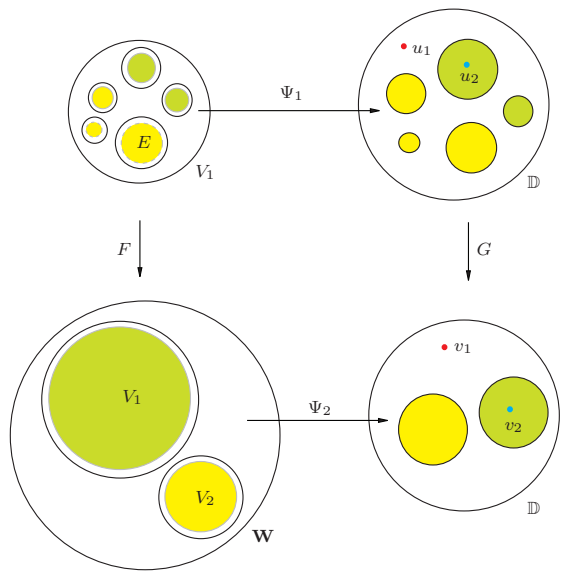}
\caption{The construction of $F$ on $V_1$. In this figure,
$q=2,i=1$.}\label{construction-1}
\end{figure}

Step 1. The first step is to construct a Blaschke product
$G:\mathbb{D}\to \mathbb{D}$ of degree $d_i$, where $\mathbb{D}$
denotes the unit disk and $d_i$ is determined below.

For every component $V_j$ of ${\bf V}$, we define
$$q_{ij}=\#\{ U\text{ a component of }{\bf U}\mid U\subset V_i,\  f(U)=V_j\}$$
Set $q_i=\max_j q_{ij}$. Then $q_i\ge 1$ and
$$\#\{\text{components of}\ {\bf U}\cap V_i\}=\sum_j q_{ij} \le q\cdot q_i\ .$$

We  construct a Blaschke product $G:\mathbb{D}\to\mathbb{D}$, as well as
a set $\mathcal D$ which is the union of $q$ Jordan domains in $\D$ with pairwise disjoint closures,
as follows:

$\bullet$ If  $V_i$ does not contain critical points of $f$, then set $G(z)=z^{q_i}$, and choose
$\mathcal D$ to be a collection of $q$ Jordan domains compactly contained in $\D\smm \{0\}$
with  pairwise disjoint closures. Set $d_i=q_i$. Note that each component of $\mathcal D$ has
exactly $q_i$ preimages. So
\begin{align*}\#\{\text{components of}\ G^{-1}(\mathcal D)\}&=q\cdot q_i\\
& \ge \#\{\text{components of}\ {\bf U}\cap V_i\} \ .
\end{align*}

$\bullet$ Otherwise, by assumption in the set-up, the set
$\mathrm{Crit}(f)$ intersects exactly one component $U$ of ${\bf
U}\cap V_i$. Set $d_i=q_i+\deg(f|_U)-1$. Choose $G$ so that it has
degree $d_i$, and has two distinct critical points $u_1$ and $u_2$
such that
  $\deg_{u_1}(G)=q_i$,  $\deg_{u_2}(G)=\deg(f|_U)$ and $G(u_1)\ne G(u_2)$.\footnote{One way
   to construct such a map $G$ is as follows: Consider the map $z\mapsto z^{q_i}$ together with a preimage
   $x\in ]0,1[$ of $1/2$. Cut $\D$ along $[x,1[$, glue in $\deg(f|_U)$ consecutive sectors
   to define a new space $\tilde D$ .
  Define a new map that maps each sector homeomorphically onto $\D\smm [\dfrac12, 1[$, and
  agrees with $z\mapsto z^{q_i}$ elsewhere.
   This gives a branched covering $\hat G$ from $\tilde D$ onto $\D$ with two critical points
   and two critical values. Use $\hat G$ to pull back the standard complex structure of $\D$
   turn  $\tilde D$in a Riemann surface. Uniformize $\tilde D$ by a map
   $\phi: \D\to \tilde D$. Then $G=\hat G\circ \phi$ suites what we need.}
Set $v_i=G(u_i)$, $i=1,2$. Now choose $\mathcal D$ to be a
collection of $q$ Jordan domains compactly contained in $\D\smm
\{v_1\}$ with  pairwise disjoint closures and with $v_2\in \mathcal
D$. Note that the preimage of any $\mathcal D$-component not
containing $v_2$ has $d_i$ components, whereas the preimage of the
$\mathcal D$-component containing $v_2$ has $d_i-\deg(f|_U)+1=q_i$
components. So
\begin{align*} \#\{\text{components of}\ G^{-1}(\mathcal D)\}
&=(q-1)d_i+q_i\\ &=(q-1)(q_i+\deg f|_U -1)+q_i\\ &=q\cdot q_i+(q-1)(\deg f|_U -1)\\ &> q\cdot q_i\\
& \ge \#\{\text{components of}\ {\bf U}\cap V_i\} \ .\end{align*}

In both cases
$G:\overline{\mathbb{D}}\smm G^{-1}(\mathcal{D})\to
\overline{\mathbb{D}}\smm \mathcal{D}$ is a   proper map
with a unique critical point.

\vspace{2mm}

Step 2. Make ${\bf U},{\bf V}$ `thick'.

In $\bf W$, take $q$ Jordan domains with smooth boundaries
$\widehat{V}_j$, $j=1,\cdots,q$, such that each $\widehat{V_j}$ is
compactly contained in $\bf W$, $V_j\subset \widehat{V}_j$ for each
$j=1,\cdots,q$, and all of the $\widehat{V}_j$ have pairwise
disjoint closures. Denote $\widehat{\bf
V}=\cup_{j=1}^q\widehat{V}_j$.

In $V_i$, take $\widehat{{\bf U}}$ to be a union of
$\#\{\text{components of }G^{-1}(\mathcal{D})\}$ (which is greater
than the number of ${\bf U}$-components in $V_i$) Jordan domains
with smooth boundaries
with the following properties:\\
-- $\widehat{{\bf U}}$ is compactly contained in $V_i$;\\
-- $({\bf U}\cap V_i)$ is compactly contained in $ \widehat{\bf U}$; \\
-- each component of $\widehat{\bf U}$ contains at most one component of $({\bf U}\cap V_i)$;\\
-- the components of $\widehat{\bf U}$ have pairwise disjoint closures.

There exists a qc map $\Psi_2: \overline{\bf W}\to \overline \D$ such that $\Psi_2(\widehat{\bf V})=\mathcal D$.

Let now $U$ be any component of $\bf U$. There is a unique component
  $\widehat U$  of $\widehat{\bf U}$ containing  $U$.
  Also $f(U)=V_j\subset \widehat V_j$ for some $j$, and $ \Psi_2(\widehat V_j)$ is a component, denoted
by $D(U)$, of $\mathcal D$. See the following diagram:
$$\begin{array}{lcc} U\subset \widehat U &  & G^{-1}(D(U))\\
\downarrow f && \downarrow G\\
V_j\subset \widehat V_j & \overset{\Psi_2}{\longrightarrow} & D(U)\end{array}$$

There is a qc map $\Psi_1: \overline V_i\to \overline \D$ so
that $\Psi_1(\widehat{\bf U})=G^{-1}(\mathcal D)$ and, for any component $U$ of $V_i\cap \bf U$,
the set $\Psi_1(\widehat{U})$ is a component of $G^{-1}(D(U))$.
Then we can define a quasiregular branched covering $
F:\overline{V}_i\smm\widehat{{\bf U}}\to \overline{{\bf
W}}\smm\widehat{{\bf V}}$ of degree $d_i$ to be $$\Psi_2^{-1}\circ G|_{\overline{\mathbb{D}}\smm G^{-1}(\mathcal{D})}\circ \Psi_1\ .$$
\vspace{2mm}

 Step 3. Glue.

Define at first $F=f$ on $V_i\cap \bf U$. For
   each component $\widehat E$  of $\widehat{{\bf U}}$ not containing a component of $\bf U$, take a Jordan domain $E$ with smooth boundary compactly contained in $\widehat{E}$. Then $F$
   maps $\partial\widehat E$ homeomorphically onto $\partial \widehat{V_j}$ for some
 $j$.
 Define $F$ to be a
conformal map from $E$ onto $V_j$ by Riemann Mapping Theorem and
$F$ extends homeomorphically from $\overline{E}$ onto
$\overline{V}_j$.

Notice that the map $F$ is defined everywhere except on a disjoint union of  annular domains, one
in each
 component of $\widehat{\bf U}$. Furthermore $F$ maps  the two boundary components
of each such annular domain  onto the boundary of $\widehat V_j\smm
V_j$ for some $j$, and is a covering of the same degree on each of
the two boundary components.

This shows that $F$ admits an extension as a covering of these annular domains.
As all boundary curves are smooth and $F$ is quasi-regular outside the annular domains,
the extension can be made quasi-regular as well.

{\bf Part II}. We may now extend $F$ to every ${\bf V}$ component
intersecting $\bf U$ following the same procedure as shown in Part
I. Assume that $V_i$ is a $\bf V$-component disjoint from $\bf U$.
We define $F:V_i\to \bf W$ to be a conformal homeomorphism and we
set $d_i=1$. We obtain a quasi-regular map $F:{\bf V}\to \bf W$ as
an extension of $f:{\bf U}\to \bf V$. By construction, $F$ is
holomorphic on $F^{-1}(\bf V)$.
\end{proof}

\REFLEM{outer} There is an integer  $d$ so that for  the map $g:
z\mapsto z^d$, the map $F$ has an extension on ${\bf W}\smm {\bf V}$
so that $F:{\bf W}\smm {\bf V}\to g({\bf W})\smm {\bf W}$ is a
quasi-regular branched covering, coincides with $g$ on $\partial
{\bf W}$ and is continuous on $\overline{{\bf W}}$. In particular
$F^{-1}({\bf W})={\bf V}$ and $F$ is holomorphic on $F^{-2}({\bf
W})=F^{-1}(\bf V)$.\ENDLEM

\begin{proof}
Set $d= \sum_{i=1}^qd_i$, where the $d_i$'s are defined in the proof of Lemma \ref{Inner}. See Figure \ref{construction-2} for
the proof of this lemma.

\begin{figure}[htbp]\centering
\includegraphics[width=8cm]{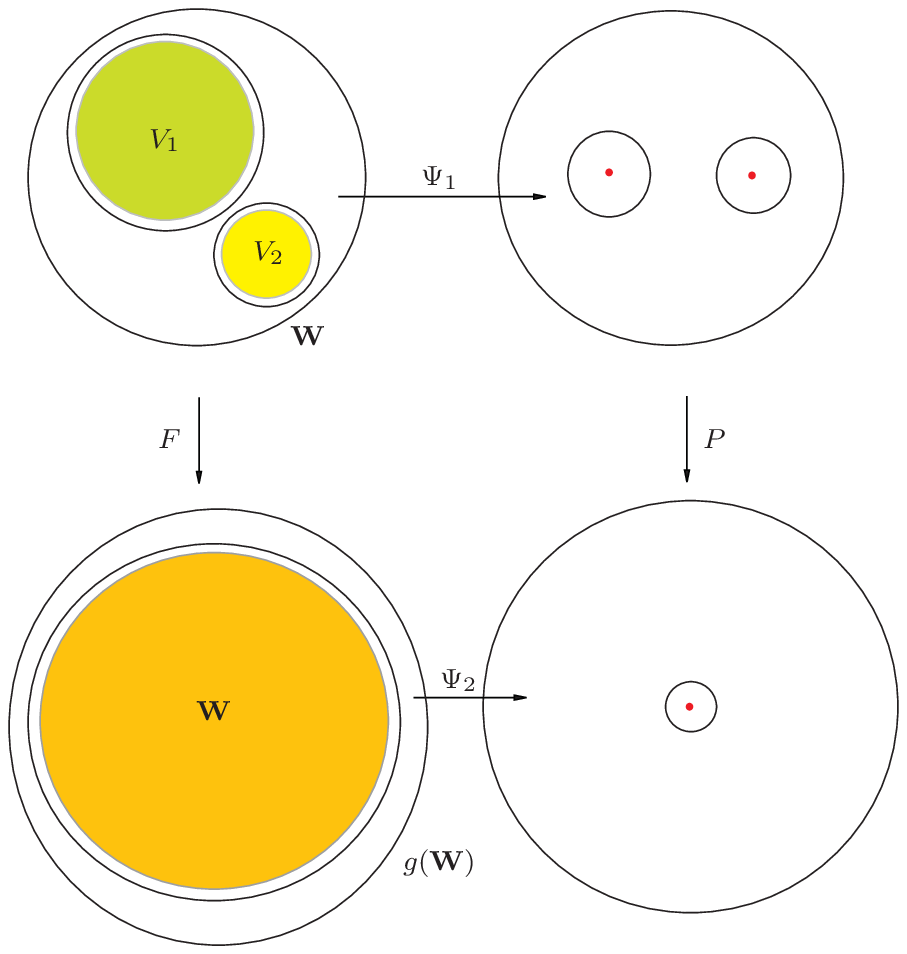}
\caption{}\label{construction-2}
\end{figure}

The domain $\widehat{{\bf V}}$ is defined as in the proof of the
previous lemma. Now take a Jordan domain $\widehat{{\bf W}}$ with
smooth boundary such that ${\bf W}\subset\subset \widehat{\bf
W}\subset\subset g({\bf W})$.

Let
$$
P(z)=(z-a_1)^{d_1}(z-a_2)^{d_2}\cdots(z-a_q)^{d_q},
$$
where $a_1,a_2,\cdots,a_q\in\mathbb{C}$ are distinct points.

Note that for each $1\le i\le q$, we have $P(a_i)=0$, and  $a_i$ is
a critical point of $P$ whenever $d_i>1$.

Take   $r>0$ small enough and $R>0$ large enough such  that
 $\{0<|z|\le r\}\cup \{R\le z|< \infty\}$ contains no   critical value of $P$. Obviously $
P:P^{-1}(\{r \le |z|\le R\})\to \{r \le |z|\le R\} $ is a
holomorphic proper map of degree $d$.

Note that $P^{-1}(\{|z|\le R\}$ is a closed Jordan domain, and the set $P^{-1}(\{|z|\le r\})$
consists of $q$ disjoint closed Jordan domains, each containing exactly one of the $a_i$'s in the
interior.

There exist qc maps $\Phi_1: \overline{\bf W}\to P^{-1}(\{|z|\le
R\})$, $\Phi_2: g(\overline{\bf W})\to \{|z|\le R\}$  such that for
$i=1,\cdots,q$, the set $\Phi_1( \widehat{V_i})$ is equal to the
component of $P^{-1}(\{|z|\le r\})$ containing $a_i$, and $\Phi_2(
\widehat{{\bf W}})= \{ |z|\le r\}$, and
$$P(\Phi_1(z))=\Phi_2(g(z)), \ z\in \partial \bf W.$$

Set $F=\Phi_2^{-1}\circ P\circ \Phi_1$ on $\overline{\bf
W}\smm\widehat{{\bf V}} $.

Fix any $i=1,\cdots,q$. Both maps $F:\partial \widehat{V_i}\to \partial \widehat{\bf W}$
and $F:\partial V_i\to \partial \bf W$ are
 coverings of degree $d_i$.
 We may thus extend as before $F$ to a qusiregular covering
map from  $\widehat V_i\smm V_i$ onto $\widehat{\bf W}\smm \bf W$.

This ends the construction of $F$.
\end{proof}

\vspace{2mm}

{\em \noindent Proof of Theorem \ref{renormalization}.} Extend the
map $F$ in Lemma \ref{outer} to $\cbar$ by setting
 $F=g$ on $\cbar\smm \bf W$.

This $F$ is  quasi-regular, and is holomorphic on $(\cbar\smm {\bf
W})\cup F^{-2}(\bf W)$. So every orbit passes at most twice the
region ${\bf W}\smm F^{-2}(\bf W)$. By Surgery Principle (see Page
130 Lemma 15 in \cite{Ah}), the map $F$ is qc-conjugate to a polynomial
$h$. The set $\KKK_F$ can be defined as for $h$ and the two
dynamical systems $F|_{\KKK_F}$ and $h|_{\KKK_h}$ are topologically
conjugate.

Theorem \ref{renormalization} holds for the pair $(h,\KKK_h)$ (in
place of $(f,\KKK_f)$\,) by Sullivan's no-wandering-domain theorem.
It follows that the result also holds for $(F, \KKK_F)$.
 But $\KKK_f$ is an $F$-invariant subset of $\KKK_F$ with every component of $\KKK_f$
 being a component of $\KKK_F$, and with $F|_{\KKK_f}=f|_{\KKK_f}$. So  the theorem holds for
 the pair $(f,\KKK_f)$.\qed

\section{Proof of Theorem \ref{rigidity} (a)}
We just repeat the  standard argument  (see for example Appendix in
\cite{Mc2}).

Assume that $f,\tilde f$ are c-equivalent. Set ${\bf U}={\bf U}_1$,
and ${\bf U}_n=f^{-n}({\bf V})$. The same objects gain a tilde for
$\tilde f$. For $t\in [0,1]$, let $h_t: \overline{\bf V} \to
\overline{\tilde {\bf V}}$ be an isotopy path linking $h_0$ to
$h_1$.

Then there is a unique continuous extension  $(t,z)\mapsto h(t,z),
[0,\infty[\times \overline{\bf V} \to \overline{\tilde {\bf V}}$
such that

0) each $h_t:z\to h(t,z)$ is a homeomorphism,

1) $h_t|_{\partial {\bf V}\cup \P}=h_0|_{\partial {\bf V}\cup \P}$,
$\forall t\in [0,+\infty[$,

2) for $n\ge 1$, $t>n$ and $x\in {\bf V}\smm {\bf U}_n$ we have
$h_t(x)=h_n(x)$,

3) for $t\in [0,1]$ the following diagram commutes:

$$\begin{array}{ccc}
\vdots && \vdots \\
{\bf U}_2 & \overset{h_{t+2}}{\longrightarrow} & \tilde {\bf U}_2\vspace{0.2cm}\\
\downarrow f && \downarrow \tilde f \vspace{0.2cm}\\
{\bf U}_1 & \overset{h_{t+1}}{\longrightarrow} & \tilde {\bf U}_1\vspace{0.2cm}\\
\downarrow f && \downarrow \tilde f \vspace{0.2cm}\\
{\bf V} & \overset{h_t}{\longrightarrow} & \tilde {\bf V}.
  \end{array}$$

  Set then $\Omega=\bigcup_{n\ge 1} {\bf V}\smm {\bf U}_n={\bf V}\smm \KKK_f$, and $\tl \Omega=\tl {\bf V}\smm \KKK_{\tl f}$.
  Then there is a qc map $H: \Omega\to \tilde \Omega$ such that
  $H(x)=h_n(x)$ for $n\ge 1$ and $x\in {\bf V}\smm {\bf U}_n$ and that
  $H\circ f|_{\Omega\cap {\bf U}}=\tilde f\circ H|_{\tilde \Omega\cap \tilde {\bf U}}$, i.e. $H$ realizes
  a qc-conjugacy from $f$ to $\tl f$ off $\KKK_f$.
  The qc constant of $H$ is equal to $C_0$, the qc constant of $h_1$ on ${\bf V}\smm {\bf U}$.

\section{Proof of Theorem \ref{rigidity} (b)}

\begin{main}
Let $f,\tl f$ be two maps in the set-up. Suppose that $H:{\bf V}\smm
\KKK_f\to \tl {\bf V}\smm \KKK_{\tl f}$ is a qc conjugacy off
$\KKK_f$. Assume that for every critical component $\KKK_f(c)$,
$c\in \mathrm{Crit}(f)$, satisfying that $f^l(\KKK_f(c))$ is a
critical periodic component of $\KKK_f$ for some $l\ge 1$ (including
the case of $\KKK_f(c)$ periodic), there are a constant $M_c$ and an
integer $N_c\ge 0$ such that for each $n\ge N_c$, the map
$H|_{\partial P_n(c)}$ admits an $M_c$-qc extension inside $P_n(c)$,
where $P_n(c)$ is a puzzle piece of depth $n$ containing $c$. Then
the property $(\ast)$ stated in Theorem \ref{assumption} holds.
\end{main}

We will postpone the proof of Main Proposition in the next section.
Here we combine this proposition and Theorem \ref{assumption} to
give a proof of Theorem \ref{rigidity} (b).

\vspace{2mm}

\noindent\textit{Proof\ of\ Theorem\ \ref{rigidity}\ (b).} By Main
Proposition and Theorem \ref{assumption}, the qc-conjugacy off
$\KKK_f$ $H$ extends to a qc-conjugacy off $\mathrm{int}\KKK_f$. By
Theorem \ref{renormalization}, every component of $\KKK_f$ with
non-empty interior is preperiodic. Under the condition that for
every preperiodic component $K$ of $\KKK_f$ with non-empty interior,
the map $H|_{\partial K}$ extends to a qc-conjugacy inside $K$, we
easily conclude that $H$ admits a qc extension across $\KKK_f$ such
that $f$ and $\tl f$ are qc-conjugate by this extension of $H$. \qed

\section{Proof of Main Proposition}
In this section, we always assume that $f$ is a map in the set-up
with the assumption that each $\bf V$-component contains at most one
critical point.

\subsection{Reduction of Main Proposition}

\begin{definition}{\label{classification}}
(1) For $x,y\in\KKK_f$, we say that the forward orbit of $x$
\textbf{combinatorially accumulates} to $y$, written as $x\to y$, if
 for any
$n\ge 0$, there is $j\ge 1$ such that $f^j(x)\in P_n(y)$.

Clearly, if $x\to y$ and $y\to z$, then $x\to z$.

Let $\mathrm{Forw}(x)=\{y\in\KKK_f\mid x\to y\}$ for $x\in\KKK_f$.

(2) Define an equivalence relation in $\mathrm{Crit}(f)$ as follows:
$$
\text{for } c_1,c_2\in \mathrm{Crit}(f), c_1\sim c_2
\Longleftrightarrow either c_1=c_2  \text{ or } (c_1\to c_2 \text{
and } c_2\to c_1).
$$

Let $[c]$ denote the equivalence class containing $c$ for $c\in
\mathrm{Crit}(f)$.

It is clear that $[c]=\{c\}$ if $c\not\to c$.

(3) We say that $[c_1]$ accumulates to $[c_2]$, written as
$[c_1]\to[c_2]$, if
$$
\exists\ c'_1\in[c_1], \exists\ c'_2\in[c_2]\text{ such that } c'_1
\to c'_2.
$$

It is easy to check that if $[c_1]\to[c_2]$, then
$$
\forall\ c''_1\in[c_1], \exists\ c''_2\in[c_2]\text{ such that }
c''_1 \to c''_2.
$$

It follows from this property that if $[c_1]\to [c_2],
[c_2]\to[c_3]$, then $[c_1]\to [c_3]$.

(4) Define $\mathcal{D}(f):=\mathrm{Crit}(f)/\sim$. Define a partial
order $\le$ in $\mathcal{D}(f)$:
$$
[c_1]\le [c_2]\Longleftrightarrow [c_1]=[c_2] \text{ or } [c_2]\to
[c_1].
$$
\end{definition}

We can decompose the quotient $\mathcal{D}(f)$ as follows. Let
$\mathcal{D}_0(f)$ be the set of elements in $\DDD(f)$ which are
minimal in the partial order $\le$, that is, $ [c]\in \DDD_0(f)$ if
and only if $[c]$ does not accumulate to any element in $\DDD(f)\smm
\{[c]\}$. For every $k\ge 0$, assume that $\DDD_k(f)$ is defined,
then $\DDD_{k+1}(f)$ is defined to be the set of elements in
$\DDD(f)$ which are minimal in the set
$\DDD(f)\setminus(\DDD_k(f)\cup\DDD_{k-1}(f)\cup \cdots\cup
\DDD_0(f))$ in the partial order $\le$.

For the construction above, we prove the properties below (refer to
Figure \ref{levels}).

\begin{figure}[htbp]\centering
\includegraphics[width=5cm]{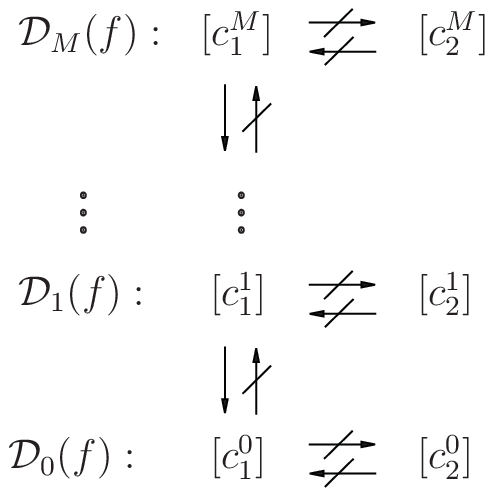}
\caption{}\label{levels}
\end{figure}

\begin{lemma}{\label{decomposition}}
(P1) There is an integer $M\ge 0$ such that
$\DDD(f)=\bigsqcup_{k=0}^{M}\DDD_k(f)$.

(P2) For every $k\ge 0$, given $[c_1],[c_2]\in \DDD_k(f),
[c_1]\ne[c_2]$, we have $[c_1]\not\to[c_2]$ and $[c_2]\not\to[c_1]$.

(P3) Let $[c_1]\in \DDD_s(f),[c_2]\in \DDD_t(f)$ with $s<t$. Then
$[c_1]\not\to[c_2]$.

(P4) For every $k\ge 1$, every $[c]$ in $\DDD_k(f)$ accumulates to
some element in $\DDD_{k-1}(f)$.
\end{lemma}

\begin{proof}
(P1) holds because $\DDD(f)$ is a finite set and
$\DDD_i(f)\cap\DDD_j(f)=\emptyset$ for $i\ne j$.

(P2) and (P3) follow directly from the minimal property of the
elements in $\DDD_k(f)$ for every $0\le k\le M$.

(P4) Let $k=1$. If there is some element $[c_1]\in\DDD_1(f)$ such
that it does not accumulate to any element in $\DDD_0(f)$, then
combining with (P2) and (P3), we have that $[c_1]$ does not
accumulate to any element in $\DDD(f)\smm\{[c_1]\}$. Consequently,
 we know $[c_1]\in \DDD_0(f)$. But $[c_1]\in\DDD_1(f)$ and
$\DDD_0(f)\cap\DDD_1(f)=\emptyset$ by (P1). We get a contradiction.
So any element in $\DDD_1(f)$ will accumulate to some element in
$\DDD_0(f)$.

Now we suppose that $k\ge 2$ and (P4) holds for
$\DDD_1(f),\DDD_2(f),\cdots, \DDD_{k-1}(f)$. Assume that (P4) is not
true for $\DDD_k(f)$, that is, there is some $[c_k]\in \DDD_k(f)$
such that $[c_k]$ does not accumulate to any element in
$\DDD_{k-1}(f)$.

If $[c_k]$ does not accumulate to any element in
$\cup_{j=0}^{k-2}\DDD_j(f)$, then by (P2) and (P3), we conclude that
$[c_k]\in \DDD_0(f)$ which contradicts the condition that $[c_k]\in
\DDD_k(f)$ and the fact that $\DDD_k(f)\cap \DDD_0(f)=\emptyset$ by
(P1).

Let $0\le i\le k-2$ be an integer satisfying that $[c_k]$ does not
accumulate to any element in $\cup_{j=i+1}^{k-1}\DDD_j(f)$ and
$[c_k]$ accumulates to some element in $\DDD_i(f)$. Then $[c_k]$
will not accumulate to any element in
$\cup_{j=i+1}^{M}\DDD_j(f)\smm\{[c_k]\}$ and hence $[c_k]\in
\DDD_{i+1}(f)$. But notice that $i+1\le k-1$ and $[c_k]\in
\DDD_k(f)$. A contradiction.
\end{proof}

Combining with the transitive property stated in Definition
\ref{classification} (3), we can consecutively apply (P4) above and
prove the following.
\begin{corollary}{\label{accumulation}}
For every $k\ge 1$, every $[c]$ in $\DDD_k(f)$ accumulates to some
element in $\DDD_0(f)$.
\end{corollary}

We will deduce Main Proposition from the following result.

\begin{proposition}{\label{equivalence}}
Let $\tl f$ be a map in the set-up. Suppose that $H:{\bf V}\smm
\KKK_f\to \tl {\bf V}\smm \KKK_{\tl f}$ is a qc-conjugacy off
$\KKK_f$. Assume that for every critical component $\KKK_f(c)$,
$c\in \mathrm{Crit}(f)$, satisfying that $f^l(\KKK_f(c))$ is a
critical periodic component of $\KKK_f$ for some $l\ge 1$ (including
the case of $\KKK_f(c)$ periodic), there are a constant $M_c$ and an
integer $N_c\ge 0$ such that for each $n\ge N_c$, the map
$H|_{\partial P_n(c)}$ admits an $M_c$-qc extension inside $P_n(c)$,
where $P_n(c)$ is a puzzle piece of depth $n$ containing $c$.
 Then for every $c\in[c_0]$ and every integer $n\ge
1$, there is a puzzle piece $K_n(c)$ containing $c$ with the
following properties.

(i) For every $c\in[c_0]$, the pieces $\{K_n(c)\}_{n\ge 1}$ is a
nested sequence.

 (ii) For each $n\ge 1$,
$\bigcup_{c\in [c_0]}K_n(c)$ is  a nice set.

(iii) There is a constant $\tl M=\tl M([c_0])$, such that for each
$n\ge 1$ and each $c\in [c_0]$, $H|_{\partial K_n(c)}$ admits an
$\tl M$-qc extension inside $K_n(c)$.
\end{proposition}

We will postpone the proof of Proposition \ref{equivalence} to the
next subsection. Here we prove the following  lemma and then use it
and Proposition \ref{equivalence} to prove Main Proposition.

\begin{lemma}{\label{unionnice}}
Let $[c_1]$ and $[c_2]$ be two distinct equivalence classes. Suppose
that for each $i=1,2$, the set $W_i$ is a nice set consisting of
finitely many puzzle pieces such that each piece contains a point in
$[c_i]$.

(1) If $[c_1]\not\to[c_2]$ and $[c_2]\not\to[c_1]$, then $W_1\cup
W_2$ is a nice set containing $[c_1]\cup [c_2]$.

(2) Suppose $[c_2]\not\to[c_1]$ and
$$
\min_{P_2 \text{ a comp. of } W_2}\mathrm{depth}(P_2) \ge \max_{P_1
\text{ a comp. of } W_1}\mathrm{depth}(P_1),
$$
i.e., the minimal depth of the components of $W_2$ is not less than
the maximal depth of those of $W_1$. Then $W_1\cup W_2$ is nice.
\end{lemma}

Before proving this lemma, we need to give an assumption for
simplicity. Notice that given two critical points $c,c'$, if
$c\not\to c'$, then there is some integer $n(c,c')$ depending on $c$
and $c'$ such that for all $j\ge 1$, for all $n\ge n(c,c')$,
$f^j(c)\not\in P_n(c')$. Since $\#\mathrm{Crit}(f)<\infty$, we can
take $n_0=\max\{n(c,c')\mid c,c'\in\mathrm{Crit}(f)\}$. Without loss
of generality, we may assume that $n_0=0$, that is to say we assume
that
$$
(\ast\ast)\ \ \ \ \ \  \text{for any two critical points\ } c, c',
\text{ for all } j\ge 1, f^j(c)\not\in P_0(c') \text{ if } c\not\to
c'.
$$
In the following paragraphs until the end of this article, we always
assume that $(\ast\ast)$ holds.

\vspace{2mm}

\noindent\emph{Proof of Lemma \ref{unionnice}.} (1) According to
Definition 2 (3) and the assumption $(\ast\ast)$, we know that
\begin{eqnarray*}
[c_1]\not\to[c_2]&\Longleftrightarrow& \forall c'_1\in [c_1],
\forall
c'_2\in [c_2], c'_1\not\to c'_2\\
&\Longleftrightarrow& \forall c'_1\in [c_1], \forall c'_2\in [c_2],
\forall n\ge 0, \forall j\ge 0, f^{j}(c'_1)\not\in P_n(c'_2)\\
&\Longleftrightarrow&  \text{For any puzzle piece } P\ni c'_1,
\forall n\ge 0, \forall j\ge 0, f^{j}(P)\cap P_n(c'_2)=\emptyset.
\end{eqnarray*}
In particular, for any $c'_1\in [c_1]$, any $c'_2\in [c_2]$, for the
component $P_1$ of $W_1$ containing $c'_1$ and the component $P_2$
of $W_2$ containing $c'_2$, for any $j\ge 0$, $f^{j}(P_1)\cap
P_2=\emptyset$. It is equivalent to say that for any component $P$
of $W_1$, for any $j\ge 0$, $f^{j}(P)\cap W_2=\emptyset$.

Similarly, from the condition $[c_2]\not\to[c_1]$, we can conclude
$f^{j}(Q)\cap W_1=\emptyset$ for any component $Q$ of $W_2$ and any
$j\ge 0$. Hence $W_1\cup W_2$ is a nice set.

(2) On one hand, from the proof of (1), we know that the condition
$[c_2]\not\to[c_1]$ implies $f^{j}(Q)\cap W_1=\emptyset$ for any
component $Q$ of $W_2$ and any $j\ge 0$.

On the other hand, for any component $P$ of $W_1$, for any $j\ge 0$,
we have
\begin{eqnarray*}
\mathrm{depth}(f^j(P))&=&\mathrm{depth}(P)-j\\
&\le& \max_{P_1 \text{ a comp. of } W_1}\mathrm{depth}(P_1)-j\\
&\le& \min_{P_2 \text{ a comp. of } W_2}\mathrm{depth}(P_2)
\end{eqnarray*}
 and then $f^j(P)$ can not be strictly contained in $W_2$.

 Hence $W_1\cup W_2$ is nice.
\qed

\vspace{2mm}

Now we can derive Main Proposition from Proposition
\ref{equivalence}.

\begin{proof}[Proof of Main Proposition]
(i) follows immediately from Proposition \ref{equivalence} (i).

(ii) For every $[\tl c]\in\DDD(f)$ and every $\hat c\in [\tl c]$,
let $\{K_n(\hat c)\}_{n\ge 1}$ be the puzzle pieces obtained in
Proposition \ref{equivalence}.

Given $[c_0]\in \DDD_k(f)$, $0\le k<M$, let
$A_k([c_0])=\{[c]\in\DDD_{k+1}(f)\mid [c]\to [c_0]\}$. Clearly,
$\#A_k([c_0])<\infty$.

Recall that $\DDD(f)=\sqcup_{i=0}^M\DDD_i(f)$. For every
$[c_0]\in\DDD_M(f)$, set $Q_n(c)=K_n(c)$ for each $c\in [c_0]$.

Now consider $[c_0]\in\DDD_{M-1}(f)$.\\
If $A_{M-1}([c_0])=\emptyset$, then set $Q_n(c)=K_n(c)$ for each
$c\in [c_0]$.\\
Otherwise, there exists a subsequence $\{l_n\}_{n\ge 1}$ of $\{n\}$
such that
$$
\min_{c'\in [c_0]}\mathrm{depth}(K_{l_n}(c'))\ge \max_{[c']\in
A_{M-1}([c_0])}\mathrm{depth}(Q_n(c'))
$$
because $\min_{c'\in [c_0]}\mathrm{depth}(K_{n}(c'))$ increasingly
tends to $\infty$ as $n\to\infty$.

We repeat this process consecutively to $\DDD_{M-2}(f), \cdots,
\DDD_0(f)$ and then all $Q_n(c)$ are defined. Combining the
properties (P2), (P3) stated in  Lemma \ref{decomposition} and Lemma
\ref{unionnice}, we easily conclude that $\cup_{c\in
\mathrm{Crit}(f)}Q_n(c)$ is a nice set for every $n\ge 1$.

(iii)  Since $\#\DDD(f)<\infty$, we can take the constant $\tl
C=\max\{\tl M([\tl c])\mid [\tl c]\in \DDD(f)\}$.
\end{proof}

\subsection{Proof of Proposition \ref{equivalence}}

First, we need to introduce a classification of the set
$\mathrm{Crit}(f)$ and several preliminary results.

\begin{definition}
(i) Suppose $c\to c$. For $c_1,c_2\in[c]$, we say that the piece
$P_{n+k}(c_1)$ is a \textbf{child} of $P_n(c_2)$ if
$f^k(P_{n+k}(c_1))=P_n(c_2)$ and $f^{k-1}:P_{n+k-1}(f(c_1))\to
P_n(c_2)$ is conformal.

The critical point $c$ is called \textbf{persistently recurrent} if
for every $n\ge 0$, every $c'\in[c]$, $P_n(c')$ has finitely many
children. Otherwise, the critical point $c$ is said to be
\textbf{reluctantly recurrent}. It is easy to check that if $c$ is
persistently recurrent (resp. reluctantly recurrent), so is every
$c'\in[c]$.

 (ii) Let
\begin{eqnarray*}
\mathrm{Crit_n}(f)&=&\{c\in \mathrm{Crit}(f)\mid c\not\to c' \text{
for any } c'\in
\mathrm{Crit}(f)\},\\
\mathrm{Crit_e}(f)&=&\{c\in \mathrm{Crit}(f)\mid c\not\to c \text{
and } \exists\
c'\in \mathrm{Crit}(f) \text{ such that } c\to c'\}, \\
\mathrm{Crit_r}(f)&=&\{c\in \mathrm{Crit}(f)\mid c\to c \text{ and }
c \text{ is
reluctantly recurrent}\},\\
\mathrm{Crit_p}(f)&=&\{c\in \mathrm{Crit}(f)\mid c\to c \text{ and }
c \text{ is persistently recurrent}\}.
\end{eqnarray*}
Then the set $\mathrm{Crit}(f)$ is decomposed into
$\mathrm{Crit_n}(f)\sqcup \mathrm{Crit_e}(f)\sqcup
\mathrm{Crit_r}(f)\sqcup \mathrm{Crit_p}(f)$.
\end{definition}

In this section, we will use sometimes the combinatorial tool -- the
tableau defined by Branner-Hubbard in \cite{bh}. The reader can also
refer to \cite{qiuyin} and \cite{fivepeople} for the definition of
the tableau.

For $x\in\KKK_f$, the tableau $\TTT(x)$ is  the graph embedded in
$\{(u,v)\mid u\in \mathbb{R}^-,v\in \R\}$ with the axis of $u$
pointing upwards and the axis of $v$ pointing rightwards (this is
the standard $\R^2$ with reversed orientation), with vertices
indexed by $-\N\times \N$, where $\N=\{0,1,\cdots\}$, with the
vertex at $(-m,0)$ being $P_m(x)$, the puzzle piece of depth $m$
containing $x$, and with $f^j(P_m(x))$ occupying the $(-m+j,j)$th
entry of $\TTT(x)$. The vertex at $(-m+j,j)$ is called
\textit{critical} if $f^j(P_m(x))$ contains a critical point. If
$f^j(P_m(x))$ contains some $y\in\KKK_f$, we call the vertex at
$(-m+j,j)$ is a $y$-vertex.

All tableau satisfy the following three basic rules (see \cite{bh},
\cite{qiuyin}, \cite{fivepeople}).

(Rule 1). In $\TTT(x)$ for $x\in\KKK_f$, if the vertex at $(-m,n)$
is a $y$-vertex, then so is the vertex at $(-i,n)$ for every $0\le i
\le m$.

(Rule 2). In $\TTT(x)$ for $x\in\KKK_f$, if the vertex at $(-m,n)$
is a $y$-vertex, then for every $0\le i\le m$, the vertex at
$(-m+i,n+i)$ is a  vertex being $P_{-m+i}(f^i(y))$.

(Rule 3) (See Figure \ref{rule3}). Given $x_1,x_2\in\KKK_f$. Suppose
that there exist integers $m_0\ge 1,n_0\ge 0,i_0\ge 1,n_1\ge 1$ and
critical points $c_1,c_2$ with the following properties.

(i) In $\TTT(x_1)$, the vertex at $(-(m_0+1),n_0)$ is a $c_1$-vertex
and $(-(m_0+1-i_0),n_0+i_0)$ is a $c_2$-vertex.

(ii) In $\TTT(x_2)$, the vertex at $(-m_0,n_1)$ is a $c_1$-vertex
and $(-(m_0+1),n_1)$ is not critical.

If in $\TTT(x_1)$, for every $0<i<i_0$, the vertex at
$(-(m_0-i),n_0+i)$ is not critical, then in $\TTT(x_2)$, the vertex
at $(-(m_0+1-i_0),n_1+i_0)$ is not critical.

\begin{figure}[h]\centering
\includegraphics[width=12cm]{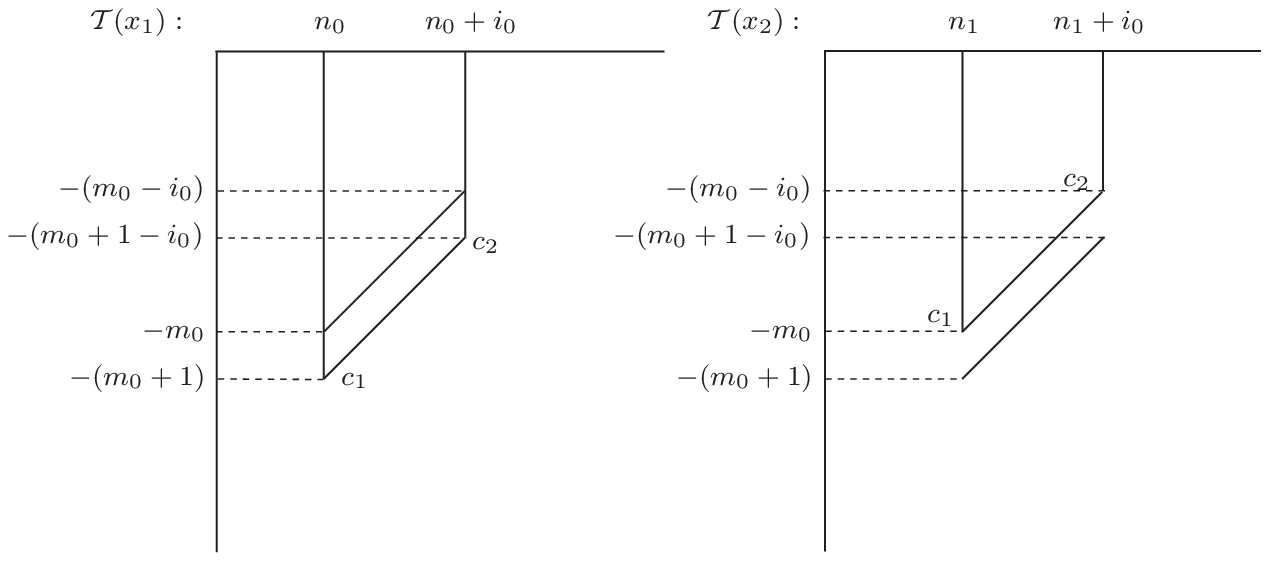}
\caption{}\label{rule3}
\end{figure}

Recall that in subsection 7.1, we made the assumption $(\ast\ast)$.
Here, we translate that assumption in the language of the tableau.
It is equivalent to assume that for $c,c'\in\mathrm{Crit}(f)$,
$c'$-vertex appears in $\TTT(c)$ iff $c'\in \mathrm{Forw}(c)$.

\begin{lemma}{\label{equation}}
1. Let $\KKK_f(c)$ be a periodic component of $\KKK_f$ with period
$p$. Then the following properties hold.

(1)
$f^i(\KKK_f(c'))\in\{\KKK_f(c),f(\KKK_f(c)),\cdots,f^{p-1}(\KKK_f(c))\}$,
$\forall c'\in \mathrm{Forw}(c)$, $\forall i\ge 0$.

(2) $\mathrm{Forw}(c)=[c]$.

(3) $c\in \mathrm{Crit_p}(f)$.

2. Let $c\in \mathrm{Crit_p}(f)$ with $\KKK_f(c)$ non-periodic. Then
the following properties hold.

(1) $\mathrm{Forw}(c)=[c]$.

(2) For every $c'\in[c]$, $c'\in \mathrm{Crit_p}(f)$ with
$\KKK_f(c')$ non-periodic.
\end{lemma}

\begin{proof}
1. Notice that $\KKK_f(c)$ is periodic iff there is a column in
$\TTT(c)\smm\{0\text{-th column}\}$ such that every vertex on that
column is a $c$-vertex. According to this and using the tableau
rules, it is easy to check that the statements in Point 1 are true.

2. (1) This property was proved by Qiu and Yin in Lemma 1,
\cite{qiuyin}. For self-containedness, we repeat their proof here.

Assume that there is some $c'\in \mathrm{Crit}(f)$ with $c\to c'$
but $c'\not\to c$. In the following, all the vertices we discuss are
in $\TTT(c)$. One may refer to Figure \ref{persistent} for the
proof.

\begin{figure}[htbp]\centering
\includegraphics[width=10cm]{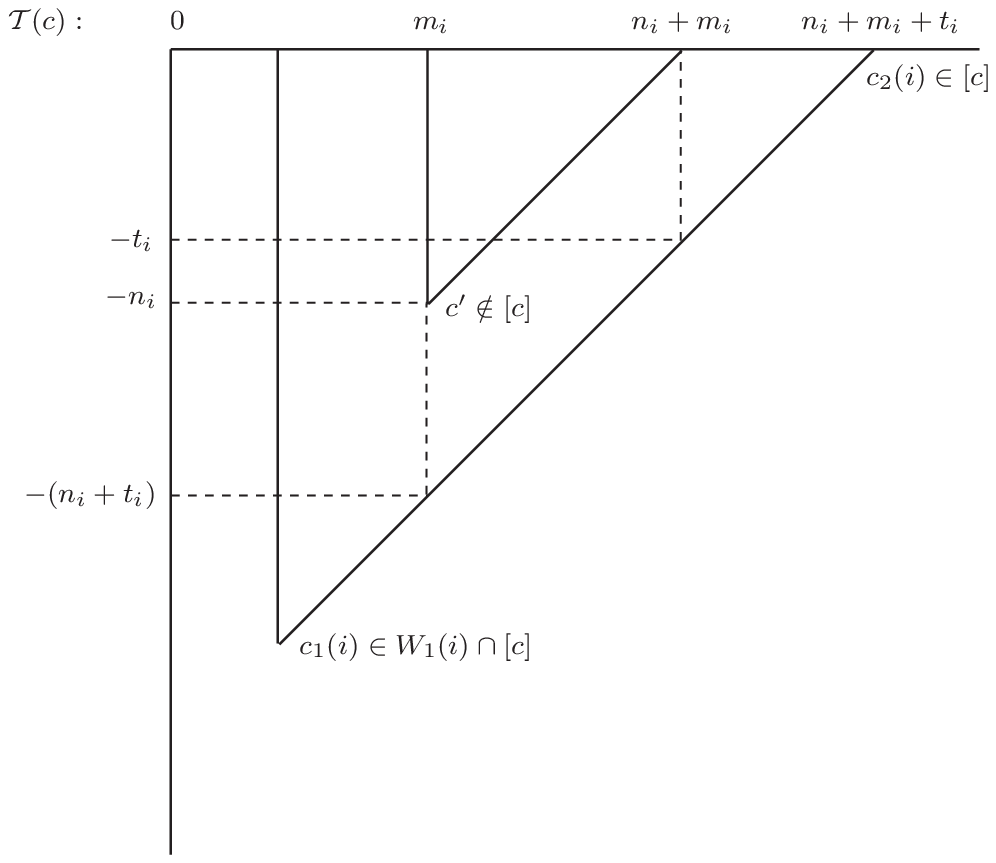}\vspace{-4mm}
\caption{}\label{persistent}
\end{figure}

If there exists a column such that every vertex on it is a
$c'$-vertex, then $c'\to c$ because $c\to c$. Hence there are
infinitely many $c'$-vertices $\{(-n_i,m_i)\}_{i\ge 1}$ such that
$(-(n_i+1), m_i)$ is not critical and $\lim_{i\to
\infty}n_i=\infty$.

By the tableau rule (Rule 2) and the assumption ($\ast\ast$), we can
see that  there are no vertices being critical points in $[c]$ on
the diagonal starting from the vertex $(-n_i,m_i)$ and ending at the
0-th row. Since $c\to c$, from the vertex $(0, n_i+m_i)$, one can
march horizontally $t_i\ge 1$ steps to the right until the first hit
of some $c_2(i)$-vertex in $[c]$. Then by (Rule 1), there are no
vertices being critical points in $[c]$ on the diagonal from the
vertex $(-t_i, n_i+m_i)$ to the vertex $(0, n_i+m_i+t_i)$.
Therefore, there are no vertices being critical points in $[c]$ on
the diagonal from the vertex $(-(n_i+t_i), m_i)$ to the vertex $(0,
n_i+m_i+t_i)$, denote this diagonal by $I$.

If there exists a point $\tl c\in \mathrm{Forw}(c)\smm [c]$ on the
diagonal $I$, then by the assumption $(\ast\ast)$, every vertex,
particularly the end vertex $(0,n_i+m_i+t_i)$ of $I$, can not be a
$\hat c$-vertex for any  $\hat c\in[c]$. This contradicts the choice
that the vertex $(0,n_i+m_i+t_i)$ is a  $c_2(i)$-vertex for
$c_2(i)\in [c]$.

Consequently, there are no critical points in $\mathrm{Forw}(c)$ on
the diagonal $I$. Combining with the assumption $(\ast\ast)$, we
know that there are no critical points on the diagonal $I$.

Follow the diagonal from the vertex $(-(n_i+t_i) , m_i)$  left
downwards until we reach a critical vertex $W_1(i)$ (such $W_1(i)$
exists since the 0-th column vertex on that diagonal is critical).
Let $c_1(i)$ be the critical point in $W_1(i)$. Then $c_1(i)\in [c]$
follows from the fact that  $(0, n_i+m_i+t_i)$ is a $c_2(i)$-vertex
for $c_2(i)\in [c]$ and the assumption $(\ast\ast)$.

Therefore, $W_1(i)$ is a child of $P_0(c_2(i))$. Notice that the
depth of $W_1(i)$ is greater than $n_i$. As $c_2(i)$ lives in the
finite set $[c]$ and $n_i\to\infty$ when $i\to \infty$, some point
in $[c]$ must have infinitely many children. This is a contradiction
with the condition $c\in\mathrm{Crit}_p(f)$.

(2) follows directly from Point 1 (1) and Point 1 (2).
\end{proof}

Set
$$
\mathrm{Crit_{per}}(f)=\{c\in \mathrm{Crit_p}(f)\mid \KKK_f(c)
\text{ is periodic}\}.
$$

\begin{lemma}{\label{Forw}}
(i) If $c\in \mathrm{Crit_n}(f)\cup \mathrm{Crit_p}(f)$, then
$[c]\in\DDD_0(f)$; if $c\in \mathrm{Crit_e}(f)$, then
$[c]\not\in\DDD_0(f)$.

(ii) For every $c_0\in \mathrm{Crit}(f)$, exactly one of the
following cases occurs.\\
Case 1. $\mathrm{Forw}(c_0)\cap(\mathrm{Crit_n}(f)\cup \mathrm{Crit_r}(f))\ne\emptyset$.\\
Case 2. $\mathrm{Forw}(c_0)\subset \mathrm{Crit_p}(f)$.\\
Case 3. For any $c\in \mathrm{Forw}(c_0)$, either (a): $c\in
\mathrm{Crit_p}(f)$ or (b): $c\in \mathrm{Crit_e}(f)$ and
$\mathrm{Forw}(c)$ contains a critical point in
$\mathrm{Crit_p}(f)$;  the critical point in (b) always exists.
\end{lemma}

\begin{proof}
(i) By the definitions of $\mathrm{Crit_n}(f)$ and
$\mathrm{Crit_e}(f)$, we easily see that if $c\in
\mathrm{Crit_n}(f)$, $[c]=\{c\}\in\DDD_0(f)$ and if $c\in
\mathrm{Crit_e}(f)$, $[c]=\{c\}\not\in\DDD_0(f)$. If $c\in
\mathrm{Crit_p}(f)$, then by the previous lemma, we know that
$\mathrm{Forw}(c)=[c]$ and then $[c]\in\DDD_0(f)$.

(ii) Suppose that neither Case 1 nor Case 2 happens. Let
$c\in\mathrm{Forw}(c_0)$ with $c\notin \mathrm{Crit_p}(f)$. Notice
that $\mathrm{Crit}(f)=\mathrm{Crit_n}(f)\cup \mathrm{Crit_r}(f)\cup
\mathrm{Crit_p}(f)\cup \mathrm{Crit_e}(f)$. So $ c\in
\mathrm{Crit_e}(f)$ and then by (i), $[c]\in \DDD_k(f)$ for some
$k\ge 1$. It follows from Corollary \ref{accumulation} that $[
c]=\{c\}$ accumulates to some element $[\tl c]\in\DDD_0(f)$.

Since Case 1 does not happen, we conclude that for every $[\hat
c]\in\DDD_0(f)$ with $[c_0]\to[\hat c]$, every point in $[\hat c]$
belongs to $\mathrm{Crit_p}(f)$. Note that $[c_0]\to [c]\to [\tl c]$
and $[\tl c]\in\DDD_0(f)$. Hence every point in $[\tl c]$ belongs to
$\mathrm{Crit_p}(f)$, particularly, $\tl c\in \mathrm{Crit_p}(f)$.
\end{proof}

Recall that in section 3, for an open set $X$ consisting of finitely
many puzzle pieces, we have defined the sets $D(X)$ and $\LLL_z(X)$
for $z\in D(X)\smm X$. The following is a property about $\LLL_z(X)$
when $X$ consists of a single piece.

\begin{lemma}{\label{nice}}
Let $P$ be a puzzle piece  and the set $\{x_1,\cdots, x_m\}\subset
{\bf V}$ be a finite set of points with each $x_i\in D(P)\smm P$ for
$1\le i\le m$. Let $f^{k_i}(\LLL_{x_i}(P))=P$ for some $k_i\ge 1$. Then\\
(1) for every $1\le i\le m$, every $0\le j<k_i$, either
$$f^j(\LLL_{x_i}(P))=\LLL_{x_s}(P)\text{ for some }1\le s\le m,
$$
or
$$
f^j(\LLL_{x_i}(P))\cap\LLL_{x_t}(P)=\emptyset \text{ for all $1\le
t\le m$};
$$
(2) $\cup_{i=1}^m\LLL_{x_i}(P)\bigcup P$ is a nice set.
\end{lemma}

\begin{proof}
(1) (by contradiction).
 Assume that there are integers $1\le i_0\le
m$, $0\le j_0< k_{i_0}$, and there is some $\LLL_{x_{i_1}}(P)$ for
$1\le i_1\le m$, such that
$$
f^{j_0}(\LLL_{x_{i_0}}(P))\ne \LLL_{x_{i_1}}(P)\text{ and }
f^{j_0}(\LLL_{x_{i_0}}(P))\cap \LLL_{x_{i_1}}(P)\ne \emptyset.
$$
Then either $f^{j_0}(\LLL_{x_{i_0}}(P))\subset\subset
\LLL_{x_{i_1}}(P)$ or $f^{j_0}(\LLL_{x_{i_0}}(P))\supset\supset
\LLL_{x_{i_1}}(P)$.

We may assume $f^{j_0}(\LLL_{x_{i_0}}(P))\subset\subset
\LLL_{x_{i_1}}(P)$. The proof of the other case is similar.

On one hand, since $f^{k_{i_0}-j_0}$ maps
$f^{j_0}(\LLL_{x_{i_0}}(P))$ onto $P$ and $f^{k_{i_1}}$ maps
$\LLL_{x_{i_1}}(P)$ onto $P$, we have $k_{i_0}-j_0>k_{i_1}$.

On the other hand, by Corollary \ref{nicecorollary} (3), we know
that $f^{j_0}(\LLL_{x_{i_0}}(P))=\LLL_{f^{j_0}(x_{i_0})}(P)$ and
then $k_{i_0}-j_0$ is the first landing time of the points in
$f^{j_0}(\LLL_{x_{i_0}}(P))$ to $P$, while from the assumption
$f^{j_0}(\LLL_{x_{i_0}}(P))\subset\subset \LLL_{x_{i_1}}(P)$, we
have that $k_{i_1}$ is also the first landing time of the points in
$f^{j_0}(\LLL_{x_{i_0}}(P))$ to $P$. So $k_{i_0}-j_0=k_{i_1}$. A
contradiction.

(2) For any $q\ge 1$ (as long as $\mathrm{depth}(f^q(P))\ge 0$), $$
\mathrm{depth}(f^q(P))<\mathrm{depth}(P)<\mathrm{depth}(\LLL_{x_s}(P))
$$ for every
$1\le s\le m$. So $f^q(P)$ can not be strictly contained in
$\cup_{i=1}^m\LLL_{x_i}(P)$ for all $q\ge 1$.

Fix $1\le i\le m$. For $1\le j< k_i$, by (1), we know that
$f^j(\LLL_{x_i}(P))$ is not strictly contained in
$\cup_{i=1}^m\LLL_{x_i}(P)$. Since $P$ is a single puzzle piece, it
is nice. By Lemma \ref{basic} (2), we have $f^j(\LLL_{x_i}(P))\cap
P=\emptyset$. When $j\ge k_i$, notice that as long as
$\mathrm{depth}(f^j(\LLL_{x_i}(P)))\ge 0$, we have
$$
\mathrm{depth}(f^j(\LLL_{x_i}(P)))\le\mathrm{depth}(P)<\mathrm{depth}(\LLL_{x_s}(P))
$$ for every
$1\le s\le m$. This implies that $f^j(\LLL_{x_i}(P))$ is not
strictly contained in $\cup_{i=1}^m\LLL_{x_i}(P)\bigcup P$.
\end{proof}

\begin{lemma}{\label{annulus}}
Let $Q,Q',P,P'$ be puzzle pieces with the following properties.

(a) $Q\subset\subset Q'$, $c_0\in P\subset\subset P'$ for $c_0\in
\mathrm{Crit}(f)$.

(b) There is an integer $l\ge 1$ such that $f^l(Q)=P, f^l(Q')=P'$.

(c) $(P'\smm P)\cap (\cup_{c\in \mathrm{Forw}(c_0)}\cup_{n\ge 0}
\{f^n(c)\})=\emptyset$.

Then for all $0\le i\le l$, $(f^i(Q')\smm f^i(Q))\cap
\mathrm{Forw}(c_0)=\emptyset$.
\end{lemma}

\begin{proof}
If $f^{l-1}(Q')\smm f^{l-1}(Q)$ contains some $c\in
\mathrm{Forw}(c_0)$, since $f(f^{l-1}(Q))=P$ and
$\deg(f:f^{l-1}(Q')\to P')=\deg_{c}(f)$, we have $f(c)\in P'\smm P$.
It contradicts the condition (c).

For the case $l=1$, the lemma holds.

Now assume $l\ge 2$.

We first prove $(f^{l-2}(Q')\smm f^{l-2}(Q))\cap
\mathrm{Forw}(c_0)=\emptyset$.

If $(f^{l-1}(Q^{\prime})\smm f^{l-1}(Q))\cap (\mathrm{Crit}(f)\smm
\mathrm{Forw}(c_0))=\emptyset$, then
$$
f:f^{l-1}(Q^{\prime})\smm f^{l-1}(Q)\rightarrow P^{\prime}\smm P.
$$
If $f^{l-2}(Q^{\prime})\smm f^{l-2}(Q)$ contains some $c'\in
\mathrm{Forw}(c_0)$, since $f(f^{l-2}(Q))=f^{l-1}(Q)$ and
$\deg(f:f^{l-2}(Q')\to f^{l-1}(Q'))=\deg_{c'}(f)$,  we have
$f(c')\in f^{l-1}(Q^{\prime})\smm f^{l-1}(Q)$ and $f^2(c')\in
P^{\prime}\smm P$. This contradicts the condition (c). Hence under
the assumption $(f^{l-1}(Q^{\prime})\smm f^{l-1}(Q))\cap
(\mathrm{Crit}(f)\smm \mathrm{Forw}(c_0))=\emptyset$, we come to the
conclusion that $(f^{l-2}(Q^{\prime})\smm f^{l-2}(Q))\cap
\mathrm{Forw}(c_0)=\emptyset$.

Otherwise, $f^{l-1}(Q^{\prime})\smm f^{l-1}(Q)$ contains some
$c_1\in \mathrm{Crit}(f)\smm \mathrm{Forw}(c_0)$. Since
$c_{1}\not\in\mathrm{Forw}(c_{0})$, $c_{1}\not\in\mathrm{Forw}(c)$
for any $c\in\mathrm{Forw}(c_{0})$. By the assumption $(\ast\ast)$,
we conclude that $(f^{l-2}(Q^{\prime})\smm f^{l-2}(Q))\cap
\mathrm{Forw}(c_{0})=\emptyset$.

Continue the similar argument as above, we could prove the lemma for
all $0\le i\le l-3$.
\end{proof}

The analytic method we will use to prove Proposition
\ref{equivalence} is the following lemma on covering maps of the
unit disk.

\REFLEM{covering} (see \cite{four} Lemma 3.2)

For every integer $d\ge 2$ and every $0<\rho <r<1$ there exists
$L_0=L_0(\rho,r,d)$ with the following property. Let $g,\tilde
g:(\D,0) \to (\D,0)$ be holomorphic proper maps of the same degree
at most $d$, with critical values contained in $\D_\rho$. Let
$\eta,\eta':\T \to \T$ be two homeomorphisms satisfying $\tilde g
\circ \eta'=\eta \circ g$, where $\T$ denotes the unit circle.
Assume that
 $\eta$ admits an $L$-qc extension $\xi:\D \to \D$ which is the identity on $\D_{r}$.
Then $\eta'$ admits an $L'$-qc extension $\xi':\D \to \D$ which is
the identity on $\D_r$, where $L'=\max \{L,L_0\}$. \ENDLEM

In the following, we will discuss $\mathrm{Crit_{per}}(f),
\mathrm{Crit_p}(f), \mathrm{Crit_n}(f)\cup \mathrm{Crit_r}(f)$ and
$\mathrm{Crit_e}(f)$ successively.

For any $c\in \mathrm{Crit_{per}}(f)$, by the condition of
Proposition \ref{equivalence}, there are a constant $M_c$ and an
integer $N_c$ such that the map $H|_{\partial P_n(c)}$ extends to an
$M_c$-qc extension inside $P_n(c)$ for all $n\ge N_c$.

The following lemma can be easily proved by Lemma \ref{nice}.
\begin{lemma}{\label{periodic nice}}
Fix a point $c_0\in \mathrm{Crit_{per}}(f)$ and set
$N:=\max\{N_c,c\in [c_0]\}$. Let $K_n(c_0)=P_{n+N}(c_0)$ and for
every $c\in[c_0]\smm \{c_0\}$, let $K_n(c)=P_{n+N+l_c}(c)$, where
$l_c$ is the smallest positive integer such that
$f^{l_c}(\KKK_f(c))=\KKK_f(c_0)$. Then $\cup_{c\in[c_0]}K_n(c)$ is
nice for every $n\ge 1$.
\end{lemma}

Set $b=\#\mathrm{Crit}(f),  \de=\max_{c\in
\mathrm{Crit}(f)}{\deg_{c}(f)} ,\text{ and }
\text{orb}_f([c_0])=\cup_{n\ge 0}\cup_{c\in [c_0]} \{f^n(c)\}$ for
$c_0\in \mathrm{Crit}(f)$.

The following theorem is one of the main results in
\cite{fivepeople}. They combined the KSS nest constructed by
Kozlovski, Shen and van Strien (\cite{kss}), and the Kahn-Lyubich
covering lemma (\cite{kahnlyubich}) to prove the theorem below.

\begin{theorem}{\label{inner}}
Given a critical point $c_0\in \mathrm{Crit_p}(f)\smm
\mathrm{Crit_{per}}(f)$. There are two constants $S$ and $\De_0>0$,
depending on $b$, $\de$ and $\widehat \mu$ (see below), and a nested
sequence of critical puzzle pieces $K_n(c_0)\subset \subset
K_{n-1}(c_0)$, $n\ge 1$, with $K_0(c_0)$ to be the critical puzzle
piece of depth $0$, satisfying that

(i) for each $K_n(c_0)$, $n\ge 1$,  we have $f^{p_n}(K_n(c_0))=
K_{n-1}(c_0)$ for some $p_n\ge 1$ and $\deg(f^{p_n}:K_n(c_0)\to
K_{n-1}(c_0))\le S$,

(ii) each $K_n(c_0)$, $n\ge 1$, contains a sub-critical piece
$K_n^-(c_0)$ such that
$$\text{\rm mod}(K_n(c_0)\smm \overline
{K_n^-(c_0)})\ge \De_0 \text{ and }(K_n(c_0)\smm \overline
{K_n^-(c_0)})\cap \mathrm{orb}_f([c_0])=\emptyset.
$$
Here
\begin{equation}\widehat \mu=\min\{ \text{\rm mod}(P_0(c_0)\smm
\overline{W})\ |\ W \text{ a component of {\bf U} contained in }
P_0(c_0)\}. \end{equation}

\end{theorem}

\begin{lemma}{\label{critp}}
Given a critical point $c_0\in \mathrm{Crit_p}(f)\smm
\mathrm{Crit_{per}}(f)$. Let $(K_n(c_0), K_n^-(c_0))_{n\ge 1}$ be
the sequence of pairs of critical puzzle pieces constructed in
Theorem \ref{inner}. For $c\in [c_0]\smm\{c_0\}$, let
$K_n(c):=\LLL_c(K_n(c_0))$. Then

(1) for every $c\in[c_0]$ and every $n\ge 1$, the restriction
$H|_{\partial K_n(c)}$ admits a qc extension inside $K_n(c)$ whose
maximal dilatation is independent of $n$;

(2) for each $n\ge 1$, $\cup_{c\in[c_0]}K_n(c)$ is nice.
\end{lemma}

\begin{proof}
(1) We first prove that $H|_{\partial K_n(c_0)}$ admits an $L'$-qc
extension inside $K_n( c_0)$ where $L'$ is independent of $n$. This
part is similar to the proof of Proposition 3.1 in \cite{PT}.

Since $H$ preserves the degree information, the puzzle piece bounded
by $H(\partial K_n(c_0))$ (resp. $H(\partial K_n^-(c_0))$) is a
critical piece for $\tilde f$, denote it by
$\tilde{K}_n(\tilde{c}_0)$ (resp. $H(\partial
\tilde{K}_n^-(\tilde{c}_0))$),
$\tilde{c}_0\in\mathrm{Crit}(\tilde{f})$.

 Notice that
$H|_{\partial K_1(c_0)}$  has a qc extension on a neighborhood of
$\partial K_1(c_0)$. It extends thus to an $L_1$-qc map
 $K_1(c_0)\to \tilde K_1(\tl c_0)$, for some $L_1\ge 1$ (see e.g. \cite{ct}, Lemma C.1).

In the construction of the sequence in Theorem \ref{inner}, the
operators $\G,\AAA,\BBB$ are used. As they can be read off from the
dynamical degree on the boundary of the puzzle pieces,
 and $H$ preserves this degree information, Theorem \ref{inner} is valid for the pair of
 sequences $(\tl K_n(\tl c_0),\tl K^-_n(\tl c_0))_{n\ge 1}$ as well,
 with the same constant  $S$, and probably a different $\tl \De_0$ as a
lower bound for $\text{\rm mod}(\tl K_n(\tl c_0)\smm \overline{\tl
K_n^-(\tl c_0)})$.

Recall that for each $i\ge 1$,  the number $p_i$ denotes the integer
such that $f^{p_i}(K_i(c_0))=K_{i-1}(c_0)$. We have  $\tl
f^{p_i}(\tl K_i(\tl c_0))=\tl K_{i-1}(\tl c_0)$. And both maps
 $f^{p_i}: K_i(c_0) \to K_{i-1}(c_0)$ and
 $\tl f^{p_i}:\tl K_i(\tl c_0) \to\tl K_{i-1}(\tl c_0)$
  are proper
 holomorphic maps of degree
$S$.

Fix now $n\ge 1$.

Set $v_n=c_0$, and then, for $i=n-1,n-2,\cdots, 1$, set
consecutively $v_i=f^{p_{i+1}+...+p_n}(c_0)$.

Since  $(K_i(c_0)\smm K^-_i(c_0))\cap
\mathrm{orb}_f([c_0])=\emptyset$, all the critical values of
$f^{p_{i+1}}|_{K_{i+1}(c_0)}$, as well as $v_i$, are contained in
$K^-_{i}(c_0)$, $1 \leq i \leq n-1$.

Let  $\psi_i:(K_i(c_0),v_i) \to (\D,0)$ be a bi-holomorphic
uniformization, $i=1,\cdots,n$.
 For $i=2,\cdots, n$, let  $g_i=\psi_{i-1} \circ f^{p_i} \circ \psi_i^{-1}$.
These maps fix the point $0$, are  proper holomorphic  maps of
degree at most $S$, with the critical values contained in
$\psi_{i-1}(K^-_{i-1}(c_0))$.

Let $\psi_i(K^-_i(c_0))=\Omega_i$. Since $\text{\rm mod}(\D
\setminus \overline{\Omega_i}) =\text{\rm mod} (K_i(c_0))\smm
\overline{K^-_i(c_0)})\geq \De_0 >0$ and $\Omega_i \ni
\psi_i(v_i)=0$, $1 \leq i \leq n$, these domains are contained in
some disk $\D_s$ with $s=s(\De_0)<1$. So the critical values of
$g_i$ are contained in $\Omega_{i-1}\subset \D_s$,
 $2 \leq i \le n$.

The corresponding objects for $\tilde f$ will be marked with a
tilde. The same assertions hold for $\tl g_i$. Then all the maps
$g_i$ and $\tl g_i$  satisfy the assumptions of Lemma
\ref{covering}, with $d\le S$, and   $\rho=\max\{s,\tl s\}$.

$$\begin{array}{ccccccc}
(\D,0) & \overset{\psi_n}{\longleftarrow} & (K_n(c_0),v_n) &
&(\tilde K_n(\tl c_0),\tilde v_n)&
\overset{\tilde \psi_n}{\longrightarrow} & (\D,0)\vspace{0.2cm}\\
g_n\downarrow  && \downarrow f^{p_n} && \tilde f^{p_n}\downarrow  &&\downarrow \tilde g_n\\
(\D,0) & \overset{\psi_{n-1}}{\longleftarrow} &
(K_{n-1}(c_0),v_{n-1}) && (\tilde K_{n-1}(\tl c_0),\tilde v_{n-1})&
\overset{\tilde \psi_{n-1}}{\longrightarrow}
& (\D,0)\vspace{0.2cm}\\
\!\!\!\!\!\!g_{n-1}\downarrow && \ \ \ \downarrow f^{p_{n-1}} &&
\!\!\!\!\!\!\tilde f^{p_{n-1}}\downarrow
&&\ \ \ \downarrow\tilde g_{n-1}\\
\ \ \ \ \vdots&&\ \!\!\!\!\!\!\!\!\!\!\!\!\vdots&&\ \ \ \ \ \vdots&&\!\!\!\!\!\!\!\vdots \\
g_{3}\downarrow && \downarrow f^{p_3} && \tilde f^{p_3}\downarrow &&\downarrow\tilde g_3\\
(\D,0) & \overset{\psi_2}{\longleftarrow} & (K_2(c_0),v_2) &
&(\tilde K_2(\tl c_0),\tilde v_2)& \overset{\tilde
\psi_2}{\longrightarrow}
& (\D,0)\vspace{0.2cm}\\
g_2\downarrow  && \downarrow f^{p_2} && \tilde f^{p_2}\downarrow  &&\downarrow \tilde g_2\\
(\D,0) & \overset{\psi_1}{\longleftarrow} & (K_1(c_0),v_1) & & (\tilde K_1(\tl c_0),\tilde v_1)&
\overset{\tilde \psi_1}{\longrightarrow} & (\D,0)\\
  \end{array}$$

  Note that each of $\psi_i,\tl\psi_i$ extends to a homeomorphism from the closure
  of the puzzle piece to $\overline \D$.

  Let us  consider homeomorphisms  $\eta_i: \T \to \T$  given by
$\eta_i=\tilde \psi_i \circ H|_{\partial K_i(c_0)} \circ
\psi_i^{-1}$. They are equivariant with respect to the $g$-actions,
i.e.,
 $\eta_{i-1} \circ g_i=\tilde g_i \circ \eta_i$.

Due to the qc extension of $H|_{\partial K_1(c_0)}$, we know that
$\eta_1$ extends to an $L_1$-qc map
 $\D \to \D$. Then $\eta_1$ is an $L_1$-quasi-symmetric  map.
 Fix some $r$ with $\rho< r< 1$. We conclude that $\eta_1$ extends to an $L$-qc map $\xi_1:\D \to \D$
which is the identity on $\D_{r}$, where $L$ depends on $L_1$,
$\rho$ and $r$.

 Let $L_0=L_0(\rho, r,S)$ be as in Lemma \ref {covering},
and let $L'=\max\{L,L_0\}$. For $i=2,3,\cdots,n$, apply
consecutively  Lemma~\ref {covering} to the following left diagram
(from bottom to top):
$$\begin{array}{rcl}
\T& \overset{\eta_n}{\longrightarrow} & \T\vspace{0.2cm}\\
g_n\downarrow  && \downarrow \tilde g_n\vspace{0.2cm}\\
\T& \overset{\eta_{n-1}}{\longrightarrow} & \T\vspace{0.2cm}\\
g_{n-1}\downarrow  && \downarrow \tilde g_{n-1}\\
\!\!\vdots &&  \vdots\\
\T& \overset{\eta_2}{\longrightarrow} & \T\vspace{0.2cm}\\
g_2\downarrow  && \downarrow \tilde g_2\vspace{0.2cm}\\
\T& \overset{\eta_1}{\longrightarrow} & \T
  \end{array},\qquad
 \text{ we get}\quad \begin{array}{ccc}
(\D,0)& \overset{\xi_n}{\longrightarrow} & (\D,0)\vspace{0.2cm}\\
g_n\downarrow  && \downarrow \tilde g_n\vspace{0.2cm}\\
(\D,0)& \overset{\xi_{n-1}}{\longrightarrow} & (\D,0)\vspace{0.2cm}\\
\!\!\!\!\!g_{n-1}\downarrow  && \ \ \ \ \downarrow \tilde g_{n-1}\\
\ \ \ \ \vdots && \!\!\!\!\!\!\!\vdots\\
(\D,0)& \overset{\xi_2}{\longrightarrow} & (\D,0)\vspace{0.2cm}\\
g_2\downarrow  && \downarrow \tilde g_2\vspace{0.2cm}\\
(\D,0)& \overset{\xi_1}{\longrightarrow} & (\D,0)
  \end{array}$$
so that for  $ i=2,\dots, n$, the map  $\eta_i$ admits an $L'$-qc
extension $\xi_i:\D \to \D$ which is the identity on $\D_r$. The
desired extension of $H|_{\partial K_n(c_0)}$ inside $K_n(c_0)$ is
now obtained by taking $\tilde \psi_n^{-1} \circ \xi_n \circ
\psi_n$.

Now we show that for $c\in[c_0]\smm\{c_0\}$, for each $n\ge 1$,
$H|_{\partial K_n(c)}$ admits an $\tl L'$-qc extension inside
$K_n(c)$ with the constant $\tl L'$ independent of $n$.

Fix $n\ge 1$.

Let $f^{q_n}(K_n(c))=K_n(c_0)$. Since $(K_n(c_0)\smm K^-_n(c_0))\cap
\mathrm{orb}_f([c_0])=\emptyset$, all the critical values of
$f^{q_n}|_{K_n(c)}$ are contained in $K^-_{n}(c_0)$.

Let  $\varphi_n:(K_n(c_0),f^{q_n}(c)) \to (\D,0)$ and
$\lambda_n:(K_n(c),c) \to (\D,0)$ be bi-holomorphic uniformizations.
Set  $\pi_n=\varphi_{n} \circ f^{q_n} \circ \lambda_n^{-1}$. This
map fixes the point $0$, is a proper holomorphic map of degree at
most $\de^b$, with the critical values contained in
$\varphi_{n}(K^-_n(c_0))$.

Since $$\text{\rm mod}(\D \setminus
\overline{\varphi_n(K^-_n(c_0))}) =\text{\rm mod} (K_n(c_0)\smm
\overline{K^-_n(c_0)})\geq \De_0 >0$$ and $\varphi_n(f^{q_n}(c))=0$
belongs to $\varphi_n(K^-_n(c_0))$, the set $\varphi_n(K^-_n(c_0))$
is contained in the disk $\D_{s}$ (here $s$ is the same number as
defined for the case of $c_0$ in this proof). So the critical values
of $\pi_n$ are contained in $\varphi_n(K^-_n(c_0))\subset \D_{s}$.

$$\begin{array}{ccccccc}
(\D,0) & \overset{\lambda_n}{\longleftarrow} & (K_n(c),c) & &(\tilde
K_n(\tl c),\tl c)&
\overset{\tilde \lambda_n}{\longrightarrow} & (\D,0)\vspace{0.2cm}\\
\pi_n \downarrow  && \downarrow f^{q_n} && \tilde f^{q_n}\downarrow  &&\downarrow \tilde \pi_n\\
(\D,0) & \overset{\varphi_n}{\longleftarrow} & (K_n(c_0),f^{q_n}(c))
& &(\tl K_n(\tl c_0),\tilde  f^{q_n}(\tl c))&
\overset{\tilde \varphi_n}{\longrightarrow} & (\D,0)\vspace{0.2cm}\\
  \end{array}$$

Let $\tl K_n(\tl c)$, $\tl c\in\mathrm{Crit}(\tl f)$ be the puzzle
piece bounded by $H(\partial K_n(c))$.  The corresponding objects
for $\tilde f$ will be marked with a tilde. The same assertions hold
for $\tl \pi_n$. Then both maps $\pi_n$ and $\tl \pi_n$  satisfy the
assumptions of Lemma \ref{covering}, with $d\le \de^b$, and
$\rho=\max\{s,\tl s\}$.

Note that each of $\varphi_n,\tl\varphi_n,\lambda_n,\tl\lambda_n$
extends to a homeomorphism from the closure
  of the puzzle piece to $\overline \D$.

  Let us  consider homeomorphisms  $\alpha_n: \T \to \T$ and $\beta_n:\T \to \T$ given by
$\alpha_n=\tilde \varphi_n \circ H|_{\partial K_n(c_0)} \circ
\varphi_n^{-1}$ and $\beta_n=\tl\lambda_n \circ H|_{\partial K_n(c)}
\circ \lambda_n^{-1}$. Then
 $\beta_n \circ \pi_n= \pi_n \circ \alpha_n$.

Due to the $L'$-qc extension of $H|_{\partial K_n(c_0)}$, we know
that $\alpha_n$ extends to an $L'$-qc map
 $\D \to \D$.
We still fix the number $r$ with $\rho< r< 1$. We can extend
$\alpha_n$ to be an $\tl L$-qc map $\mu_n:\D \to \D$ which is the
identity on $\D_r$, where $\tl L$ depends on $L'$, $\rho$ and $r$.

 Let $\tl L_0=\tl L_0(\rho, r,\de^b)$ be as in Lemma \ref {covering},
and let $\tl L'=\max\{\tl L,\tl L_0\}$. We apply Lemma~\ref
{covering} to the following left diagram:
$$\begin{array}{rcl}
\T& \overset{\beta_n}{\longrightarrow} & \T\vspace{0.2cm}\\
\pi_n\downarrow  && \downarrow \tilde \pi_n\vspace{0.2cm}\\
\T& \overset{\alpha_n}{\longrightarrow} & \T\vspace{0.2cm}\\
\end{array},\qquad
 \text{ we get}\quad \begin{array}{ccc}
(\D,0)& \overset{\nu_n}{\longrightarrow} & (\D,0)\vspace{0.2cm}\\
\pi_n\downarrow  && \downarrow \tilde \pi_n\vspace{0.2cm}\\
(\D,0)& \overset{\mu_n}{\longrightarrow} & (\D,0)\vspace{0.2cm}\\
 \end{array}$$
so that the map  $\be_n$ admits an $\tl L'$-qc extension $\nu_n:\D
\to \D$ which is the identity on $\D_r$. The desired extension of
$H|_{\partial K_n(c)}$ inside $K_n(c)$ is  obtained by taking
$\tilde \lambda_n^{-1} \circ \nu_n \circ \lambda_n$.

(2) The fact that the set $\cup_{c\in[c_0]}K_n(c)$ is nice follows
directly from Lemma \ref{nice}.
\end{proof}

\begin{lemma}{\label{critn}}
Given $c_0\in \mathrm{Crit_n}(f)\cup \mathrm{Crit_r}(f)$. Then there
exist a puzzle piece $P$ of depth $n_0$, a topological disk
$T\subset\subset P$,
 and for each $c\in
\mathrm{Crit}(f)$ with $c=c_0$ or $c\to c_0$, there is a nested
sequence of puzzle pieces containing $c$, denoted by
$\{K_n(c)\}_{n\ge 1}$, satisfying the following properties.

(1) Every $K_n(c)$ is a pullback of $P$, that is $f^{s_n}(K_n(c))=P$
for some $s_n\ge 1$. Furthermore, $\mathrm{deg}(f^{s_n}:K_n(c)\to
P)\leq \de ^{b+1}$ and all critical values of the map
$f^{s_n}:K_n(c)\to P$ are contained in $T$.

(2) $H|_{\partial K_n(c)}$ admits a qc extension inside $K_n(c)$
with a maximal dilatation independent of $n$.

(3) For every $n\ge 1$, $\bigcup_{c\in[c_0]}K_n(c)$ is a nice set.

\end{lemma}

\begin{proof}
(1) Suppose $c_0\in \mathrm{Crit_n}(f)$ and then $[c_0]=\{c_0\}$. In
$\TTT(c_0)\smm \{0\text{-th column}\}$, every vertex  is
non-critical. So for each $n\ge 1$,
\begin{eqnarray*}
\deg(f^n:P_n(c_0)\to P_0(f^n(c_0)))=\deg_{c_0}(f)\le\de.
\end{eqnarray*}
Since there are finitely many puzzle pieces of the same depth, we
can take a subsequence $\{u_n\}_{n\ge 1}$ such that
$f^{u_n}(P_{u_n}(c_0))=P$ for some fixed puzzle piece $P$ of depth
0.

Given $c\to c_0,c\in \mathrm{Crit}(f)$, let
$f^{v_n}(\LLL_c(P_{u_n}(c_0))=P_{u_n}(c_0)$. Then
\begin{eqnarray*}
& &\deg(f^{v_n+u_n}:\LLL_c(P_{u_n}(c_0))\to P)\\
&=&\deg(f^{v_n}:\LLL_c(P_{u_n}(c_0))\to P_{u_n}(c_0))\cdot
\deg(f^{u_n}:P_{u_n}(c_0)\to P)\\
&\le& \de^b\cdot\de
\end{eqnarray*}

  For $c_0\in \mathrm{Crit_n}(f)$, we set $K_n(c_0)=P_{u_n}(c_0)$ and $s_n=u_n$.
  For $c\to c_0,c\in \mathrm{Crit}(f)$, set $K_n(c)=\LLL_c(P_{u_n}(c_0))$
  and $s_n=v_n+u_n$.

Now suppose $c_0\in \mathrm{Crit_r}(f)$ and $c\to c_0, c\in
\mathrm{Crit}(f)$. Since $\TTT(c_0)$ is reluctantly recurrent, there
exist an integer $n_0\ge 0$, $c_1,c_2\in [c_0]$ and infinitely many
integers $l_n\ge 1$ such that $\{P_{n_0+l_n}(c_2)\}_{n\ge 1}$ are
children of $P_{n_0}(c_1)$ and then
$\deg(f^{l_n}:P_{n_0+l_n}(c_2)\to
P_{n_0}(c_1))=\deg_{c_2}(f)\le\de$.  Suppose
$f^{k_n}(\LLL_c(P_{n_0+l_n}(c_2)))=P_{n_0+l_n}(c_2)$ for $k_n\ge 1$.
Then
\begin{eqnarray*}
&\ &\deg(f^{l_n+k_n}:P_{n_0+l_n+k_n}(c)\to P_{n_0}(c_1))\\
&=&\deg(f^{k_n}:\LLL_c(P_{n_0+l_n}(c_2))\to
P_{n_0+l_n}(c_2))\cdot\deg(f^{l_n}:P_{n_0+l_n}(c_2)\to
P_{n_0}(c_1))\\
&\le& \de^b\cdot\de.
\end{eqnarray*}
Take a strictly increasing subsequence of $l_n$, still denoted by
$l_n$, such that $\{P_{n_0+l_n}(c_2)\}_{n\ge 1}$ is a nested
sequence of puzzle pieces containing $c_2$ and then
$\{P_{n_0+l_n+k_n}(c)\}_{n\ge 1}$ is a nested sequence of puzzle
pieces containing $c$. Set $P=P_{n_0}(c_1)$. For $c=c_2$, set
$s_n=l_n$ and $K_n(c)=P_{n_0+l_n}(c)$; for $c\ne c_2$, set
$s_n=k_n+l_n$ and $K_n(c)=P_{n_0+k_n+l_n}(c)$.

Now fix $c_0\in \mathrm{Crit_n}(f)\cup \mathrm{Crit_r}(f)$. Take a
topological disk $T\subset\subset P$ such that $T$ contains all the
puzzle pieces of depth $n_0+1$. For each $c=c_0$ or $c\to c_0$,
$c\in\mathrm{Crit}(f)$, and each $n\ge 1$, all critical values of
$f^{s_n}|_{K_n(c)}$ are contained in the union of puzzle pieces of
depth $n_0+1$ because $\mathrm{Crit}(f)\subset \KKK_f$ and
$f(\KKK_f)=\KKK_f$. Consequently, the set $T$ contains all the
critical values of the map $f^{s_n}|_{K_n(c)}$.

(2) We will use Lemma \ref{covering} to construct the qc extension.

Fix $c=c_0$ or $c\to c_0$ and $n\ge 1$.

Let $\tilde{P}$ and $\tilde K_n(\tilde c)$ be the puzzle pieces for
$\tl f$ bounded by $H(\partial P)$ and $H(\partial K_n(c))$
respectively. Since $H$ preserves the degree information, for the
map $\tilde{f}$, we also have the similar statement as for $f$ in
(1), more precisely, $\tilde f^{s_n}(\tilde K_n(\tilde c))=\tilde
P$, $\mathrm{deg}(\tl f^{s_n}:\tilde K_n(\tilde c)\to \tilde P)\le
\de^{b+1}$, and all critical values of the map $\tilde
f^{s_n}:\tilde K_n(\tilde c)\to \tilde P$ are contained in $\tilde
T$, where $\tilde T$ is a topological disk in $\tilde P$ containing
all puzzle pieces for $\tilde f$ of depth $n_0+1$ in $\tilde P$.

Let $\text{mod}(P\smm \overline{T})=\Delta_1$ and
$\text{mod}(\tilde{P}\smm \overline{\tilde{T}})=\tilde{\Delta}_1$.

 Let $\iota_n:(P,f^{s_n}(c))\to (\D,0)$ and
$\theta_n:(K_n(c),c) \to (\D,0)$ be bi-holomorphic uniformizations.
Let $h_n=\iota_n \circ f^{s_n} \circ \theta_n^{-1}$. Then $h_n$
fixes the point $0$, is a proper holomorphic map of degree at most
$\de ^{b+1}$, with all the critical values contained in
$\iota_n(T)$.

Since $\text{\rm mod}(\D \setminus \overline{\iota_n(T)}) =\text{\rm
mod} (P\smm \overline{T})= \Delta_1
>0$ and $\iota_n(T) \ni \iota_n(f^{s_n}(c))=0$, we have
$\iota_n(T)\subset\D_t$ with $t=t( \De_1)<1$. So the critical values
of  $h_n$ are contained in $\iota_n(T)\subset \D_t$.

The corresponding objects for $\tilde f$ will be marked with a
tilde. The same assertions hold for $\tl h_n$. Then  the maps $h_n$
and $\tl h_n$  satisfy the assumptions of Lemma \ref{covering}, with
$d\le \de^{b+1}$, and $\rho=\max\{t,\tl t\}$.

$$\begin{array}{ccccccc}
(\D,0) & \overset{\theta_n}{\longleftarrow} & (K_n(c),c) & &(\tilde
K_n(\tl c),\tl c)&
\overset{\tilde \theta_n}{\longrightarrow} & (\D,0)\vspace{0.2cm}\\
h_n \downarrow  && \downarrow f^{s_n} && \tilde f^{s_n} \downarrow  &&\downarrow \tilde h_n\\
(\D,0) & \overset{\iota_n}{\longleftarrow} & (P,f^{s_n}(c)) &
&(\tilde P,\tilde  f^{s_n}(\tl c))&
\overset{\tilde \iota_n}{\longrightarrow} & (\D,0)\vspace{0.2cm}\\
  \end{array}$$

Note that each of $\iota_n,\tl \iota_n, \theta_n,\tl\theta_n$
extends to a homeomorphism from the closure of the puzzle piece to
$\overline \D$.

Let us consider homeomorphisms  $\kappa_n: \T \to \T$ and $\sigma_n:
\T \to \T$  given by $\kappa_n=\tilde \iota_n \circ H|_{\partial P}
\circ \iota_n^{-1}$ and $\sigma_n=\tilde \theta_n \circ H|_{\partial
K_n(c)} \circ \theta_n^{-1}$ respectively. Then $\kappa_n \circ
h_n=\tilde h_n \circ \sigma_n$.

Notice that $H|_{\partial P}$ has a $K_1$-qc extension in $P$ for
some $K_1\ge 1$. The number $K_1$ is independent of $n$ because the
choice of $P$ does not depend on $n$.
 Fix some $r$ with $\rho< r< 1$.
We conclude that  $\kappa_n$ extends to a $K$-qc map $\omega_n:\D
\to \D$ which is the identity on $\D_{r}$, where $K$ depends on
$K_1$, $\rho$ and $r$.

 Let $K_0=K_0(\rho, r,\de^{b+1})$ be as in Lemma \ref {covering}
and let $K'=\max\{K,K_0\}$. Apply Lemma~\ref {covering} to the
following left diagram :
$$\begin{array}{rcl}
\T& \overset{\sigma_n}{\longrightarrow} & \T\vspace{0.2cm}\\
h_n\downarrow  && \downarrow \tilde h_n\vspace{0.2cm}\\
\T& \overset{\kappa_n}{\longrightarrow} & \T\vspace{0.2cm}\\
\end{array},\qquad
 \text{ we get}\quad \begin{array}{ccc}
(\D,0)& \overset{\zeta_n}{\longrightarrow} & (\D,0)\vspace{0.2cm}\\
h_n\downarrow  && \downarrow \tilde h_n\vspace{0.2cm}\\
(\D,0)& \overset{\omega_n}{\longrightarrow} & (\D,0)\vspace{0.2cm}\\
 \end{array}$$
so that the map  $\sigma_n$ admits a $K'$-qc extension $\zeta_n:\D
\to \D$ which is the identity on $\D_r$. The desired extension of
$H|_{\partial K_n(c)}$ inside $K_n(c)$ is now obtained by taking
$\tilde \theta_n^{-1} \circ \zeta_n \circ \theta_n$.

(3) Fix $n\ge 1$.

For $c_0\in \mathrm{Crit_n}(f)$, since $[c_0]=\{c_0\}$, we know that
$\ds\cup_{c\in[c_0]}K_n(c)=K_n(c_0)$ and obviously $K_n(c_0)$ is
nice.

For the case $c_0\in \mathrm{Crit_r}(f)$, we apply Lemma \ref{nice}
to
$$
\ds\bigcup_{c\in[c_0]}K_n(c)=\cup_{c\in[c_0]\smm
\{c_2\}}\LLL_c(P_{n_0+l_n}(c_2))\bigcup P_{n_0+l_n}(c_2)
$$
and easily get the conclusion.
\end{proof}

\begin{lemma}{\label{crite}}
Suppose $c_0\in \mathrm{Crit_e}(f)$. Then

(1) there is a nested sequence of puzzle pieces containing $c_0$,
denoted by $\{K_n(c_0)\}_{n\ge 1}$, such that for each $n\ge 1$,
$H|_{\partial K_n(c_0)}$ admits a qc extension inside $K_n(c_0)$
with a maximal dilatation independent of $n$,

(2) for every $n\ge 1$, $\bigcup_{c\in[c_0]}K_n(c)$ is a nice set.
\end{lemma}
\begin{proof}
Suppose $c_0\in \mathrm{Crit_e}(f)$.  In the following, we will
discuss the three cases indicated in Lemma \ref{Forw} (ii).

In Case 1, i.e., $\mathrm{Forw}(c_0)\cap
(\mathrm{Crit_n}(f)\cup\mathrm{Crit_r}(f))\ne\emptyset$, by using
Lemma \ref{critn}, we can get a nested sequence of puzzle pieces
containing $c_0$, denoted by $\{K_n(c_0)\}_{n\ge 1}$, and
$H|_{\partial K_n(c)}$ admits a qc extension inside $K_n(c_0)$ whose
maximal dilatation is independent of $n$.

We divide Case 2 ($\mathrm{Forw}(c_0)\subset \mathrm{Crit_p}(f)$)
into two subcases.

Subcase 1. There is a critical point $c_1\in \mathrm{Forw}(c_0)\cap
(\mathrm{Crit_p}(f)\smm \mathrm{Crit_{per}}(f))$. Let
$(K_n(c_1),K_n^-(c_1))_{n\ge 1}$ be the sequence of pairs of
critical puzzle pieces constructed in Theorem \ref{inner}.

For $n\ge 1$, set
$$K_n^-(c_0)=\LLL_{c_0}(K_n^-(c_1)),
f^{r_n}(K_n^-(c_0))=K_n^-(c_1) \text{ and }
K_n(c_0)=\mathrm{Comp}_{c_0}(f^{-r_n}(K_n(c_1))).
$$
Clearly, $(K_n(c_0)\smm K_n^-(c_0))\cap \mathrm{Crit}(f)=\emptyset$.
Since $(K_n(c_1)\smm K_n^-(c_1))\cap
\mathrm{orb}_f([c_1])=\emptyset$ and $\mathrm{Forw}(c_1)=[c_1]$, by
Lemma \ref{annulus}, we conclude that for all $1\le i< r_n$,
$(f^i(K_n(c_0))\smm f^i(K_n^-(c_0)))\cap
\mathrm{Forw}(c_1)=\emptyset$.

We claim that for every $n\ge 1$ and every $1\le i< r_n$,
$$
(f^i(K_n(c_0))\smm f^i(K_n^-(c_0)))\bigcap (\mathrm{Crit}(f)\smm
\mathrm{Forw}(c_1))=\emptyset.
$$ If not, there is some
$n$ such that
$$
\{f(K_n(c_0))\smm f(K_n^-(c_0)),\cdots, f^{r_n-1}(K_n(c_0))\smm
f^{r_n-1}(K_n^-(c_0))\}
$$ meets some
critical point, say $c_2\in \mathrm{Crit}(f)\smm
\mathrm{Forw}(c_1)$. See Figure \ref{case2}.

\begin{figure}[htbp]\centering
\includegraphics[width=9cm]{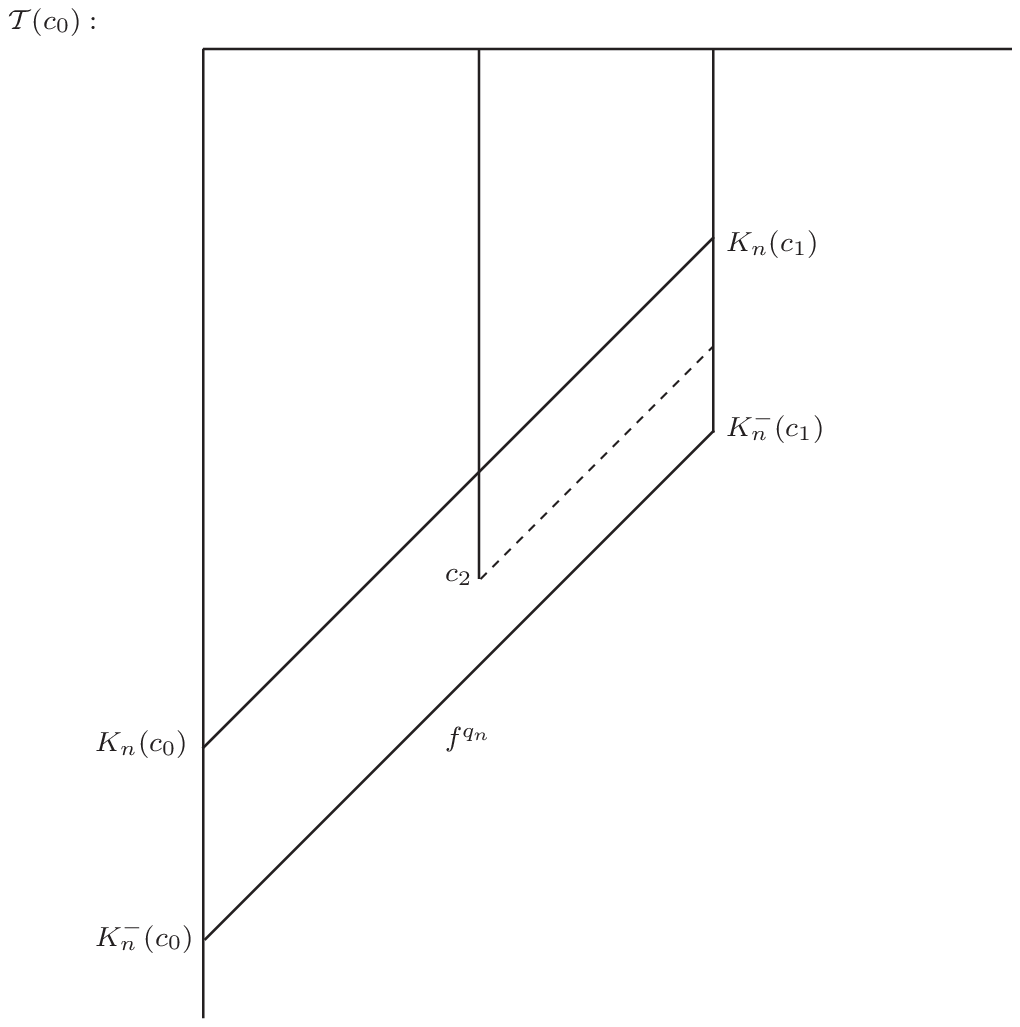}\vspace{-4mm}\caption{}
\label{case2}
\end{figure}

Then $c_0\to c_2\to c_1$. Since $\mathrm{Forw}(c_0)\subset
\mathrm{Crit_p}(f)$, we have $c_2\in \mathrm{Crit_p}(f)$ and then
$\mathrm{Forw}(c_2)=[c_2]$. So $c_1\to c_2$. It contradicts
$c_2\not\in \mathrm{Forw}(c_1)$.

Hence for every $n\ge 1$ and every $0\le i< r_n$,
\begin{eqnarray}{\label{emptyset}}
(f^{i}(K_n(c_0))\smm f^{i}(K_n^-(c_0)))\bigcap
\mathrm{Crit}(f)=\emptyset.
\end{eqnarray}

From the equation (\ref{emptyset}) and
$K_n^-(c_0)=\LLL_{c_0}(K_n^-(c_1))$, we conclude that for $n\ge 1$,
$$
\deg(f^{r_n}:K_n(c_0)\to K_n(c_1))\le \de^b.
$$
Again by the equation (\ref{emptyset}), we know that all critical
values of the map $f^{r_n}:K_n(c_0)\to K_n(c_1)$ are contained in
$K_n^-(c_1)$.

Using the similar method as in the proof of Lemma \ref{critp} (1),
we could obtain a qc extension of $H|_{\partial K_n(c_0)}$ inside
$K_n(c_0)$ whose maximal dilatation is independent of $n$. We omit
the details here.

Subcase 2. Suppose $\mathrm{Forw(c_0)}\subset
\mathrm{Crit_{per}}(f)$.

If $f^l(\KKK_f(c_0))$ is periodic for some $l\ge 1$, then there is
some critical periodic component in the periodic cycle of
$f^l(\KKK_f(c_0))$. By the condition of Proposition
\ref{equivalence}, there is an integer $N_{c_0}$ such that
$H|_{\partial P_{N_{c_0}+n}(c_0)}$ has an $M_{c_0}$-qc extension,
where $M_{c_0}$ is independent of $n$. Set
$K_n(c_0):=P_{N_{c_0}+n}(c_0)$. We are done.

Now we suppose that $\KKK_f(c_0)$ is wandering. For each $\hat c\in
\mathrm{Forw(c_0)}$, by the condition of Proposition
\ref{equivalence}, there are a constant $M_{\hat c}$ and an integer
$N_{\hat c}$ such that the map $H|_{\partial P_n(\hat c)}$ extends
to an $M_{\hat c}$-qc extension inside $P_n(\hat c)$ for all $n\ge
N_{\hat c}$. Set $N:=\max\{N_{\hat c},\hat c\in
\mathrm{Forw(c_0)}\}$.

We claim that
\begin{claim}
 There exist a point $c_1\in \mathrm{Forw(c_0)}$, a
topological disk $T_1\subset\subset P_N(c_1)$ and a nested sequence
of puzzle pieces containing $c_0$, denoted by $\{K_n(c_0)\}_{n\ge
1}$, satisfying that for every $n\ge 1$,
$f^{w_n}(K_n(c_0))=P_N(c_1)$ for some $w_n\ge 1$, $\deg
(f^{w_n}:K_n(c_0)\to P_N(c_1))\le \de^b$ and all critical values of
the map $f^{w_n}|_{K_n(c_0)}$ are contained in the set $Z$.
\end{claim}

\begin{proof}
Suppose $c\in\mathrm{Forw}(c_0)$.  Refer to the following figure for
the proof.

\begin{figure}[htbp]\centering
\includegraphics[width=10cm]{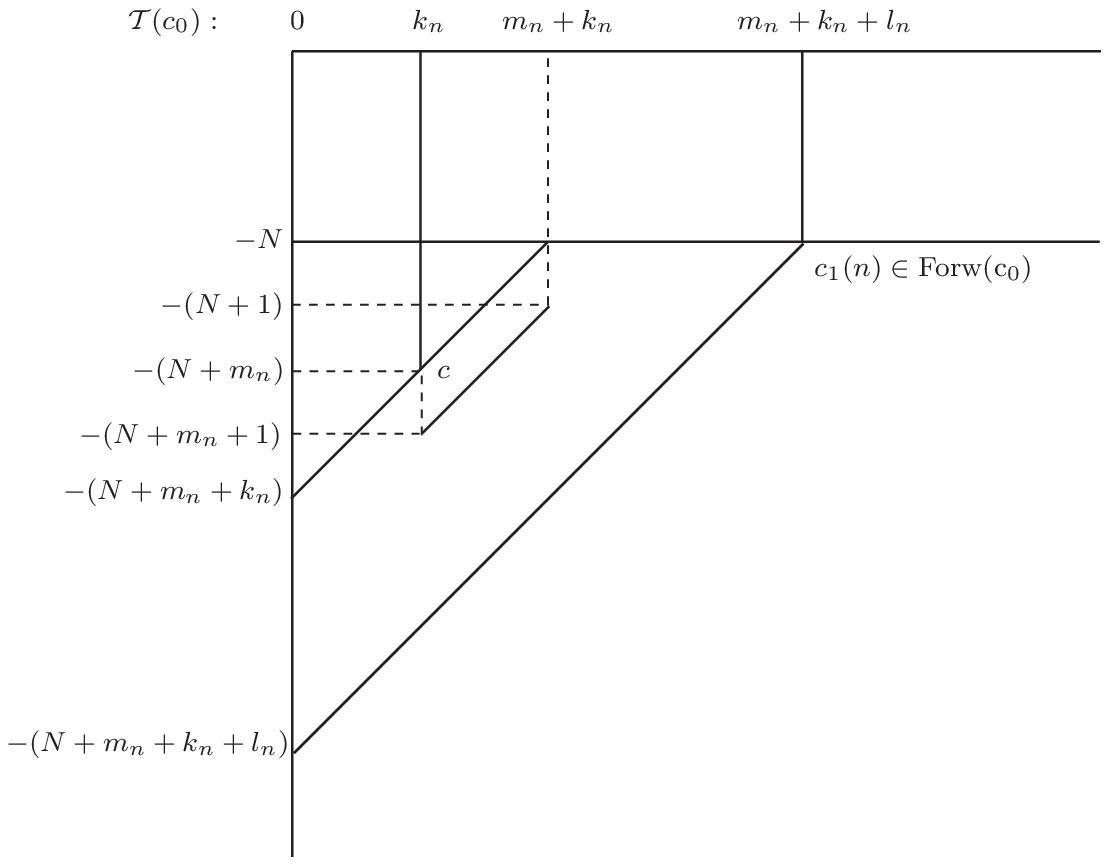}
\label{periodic}
\end{figure}

Since $\KKK_f(c_0)$ is wandering, in $\TTT(c_0)$, there are
infinitely many vertices $\{(-(N+m_n),k_n)\}_{n\ge 1}$ such that
$(-(N+m_n),k_n)$ is the first vertex being $c$ on the $m_n$-th row,
$(-(N+m_n+1),k_n)$ is not critical and
$\lim_{n\to\infty}m_n=\infty$. Then
$$
f^{k_n}(\LLL_{c_0}(P_{N+m_n}(c)))=P_{N+m_n}(c)
$$ and
$\deg (f^{k_n}:P_{N+m_n+k_n}(c_0)\to P_{N+m_n}(c))\le \de^b$.

Let $p$ be the period of $\KKK_f(c)$. Then in $\TTT(c)$, for every
$0<j<p$, either $(-(N+m_n-j),j)$ is not critical  or
$(-(N+m_n+1-j),j)$ is critical. Using (Rule 3) several times, we
conclude that in $\TTT(c_0)$, there are no critical vertices on the
diagonal starting from the vertex at $(-(N+m_n+1),k_n)$ to the
vertex at $(-(N+1),m_n+k_n)$. From the vertex $(-N,m_n+k_n)$, march
horizontally $l_n\ge 1$ steps until the first hit of some $c_1(n)$
vertex for some $c_1(n)\in \mathrm{Forw}(c_0)$. Then in $\TTT(c_0)$,
there is no critical vertex on the diagonal starting from the vertex
$(-(N+m_n+k_n+l_n-1),1)$ to the vertex $(-(N+1),m_n+k_n+l_n-1)$.
Therefore
$$
\deg(f^{m_n+k_n+l_n}:P_{N+m_n+k_n+l_n}(c_0)\to P_N(c_1(n)))\le
\de^b.
$$

Since $c_1(n)$ belongs to the finite set $\mathrm{Forw(c_0)}$ and
$m_n\to\infty$ as $n\to \infty$, we could find a subsequence of $n$,
say itself, such that $\{P_{N+m_n+k_n+l_n}(c_0)\}_{n\ge1}$ form a
nested sequence and $c_1(n)\equiv c_1$. Set $w_n:=m_n+k_n+l_n$ and
$K_n(c_0):=P_{N+w_n}(c_0)$.

Similarly to the proof of Lemma \ref{critn} (1), one takes a
topological disk $T_1\subset\subset P_N(c_1)$ such that all the
puzzle pieces of depth $N+1$ are contained in $T_1$ and
particularly, all of the critical values of $f^{w_n}|_{K_n(c_0)}$
are contained in $T_1$.
\end{proof}
Using the similar method in the proof of Lemma \ref{critn} (2), we
could show that $H|_{\partial K_n(c_0)}$ admits a qc extension
inside $K_n(c_0)$ whose maximal dilatation is independent of $n$.

In Case 3, we will first draw the similar conclusion to Lemma
\ref{critn} (1).

Take arbitrarily a point $c\in
\mathrm{Crit_{e}}(f)\cap\mathrm{Forw}(c_0)$. In $T(c_0)$, let
$\{(0,t_n)\}_{n\ge 1}$ be all the $c$-vertices on the 0-th row with
$1\le t_1< t_2<\cdots$.

Since $c_0\in \mathrm{Crit_e}(f)$ and then $c_0\not\to c_0$, by the
assumption $(\ast\ast)$, the $c_0$-vertex will not appear in
$\TTT(c_0)\smm\{$0-th column$\}$. In particular, for each $n\ge 1$,
there are no $c_0$-vertices on the diagonal starting from the vertex
$(-(t_n-1),1)$ and ending at the vertex $(-1, t_n-1)$. Denote that
diagonal by $J_n$. Since $c\not\to c$, by the assumption
$(\ast\ast)$, there are no $c$-vertices on the diagonal $J_n$.

We claim that for every $n\ge 1$, the diagonal $J_n$ meets every
point in $\mathrm{Crit}(f)\smm\{c_0,c\}$ at most once. In fact, if
not, then there is some $n'$ and some $c'\in \mathrm{Crit}(f)$ such
that the diagonal $J_{n'}$ meets $c'$ at least twice. By the
assumption $(\ast\ast)$, we can conclude that $c_0\to c'\to c'\to
c$. See Figure \ref{case3}.  By the condition of Case 3 and $c'\to
c'$, we know that $c'\in \mathrm{Crit_p}(f)$ and then
$\mathrm{Forw}(c')=[c']$ by Lemma \ref{equation}. Thus $c\to c'$ and
then $c\to c$. This contradicts $c\in \mathrm{Crit_e}(f)$.

\begin{figure}[h]\centering
\includegraphics[width=9cm]{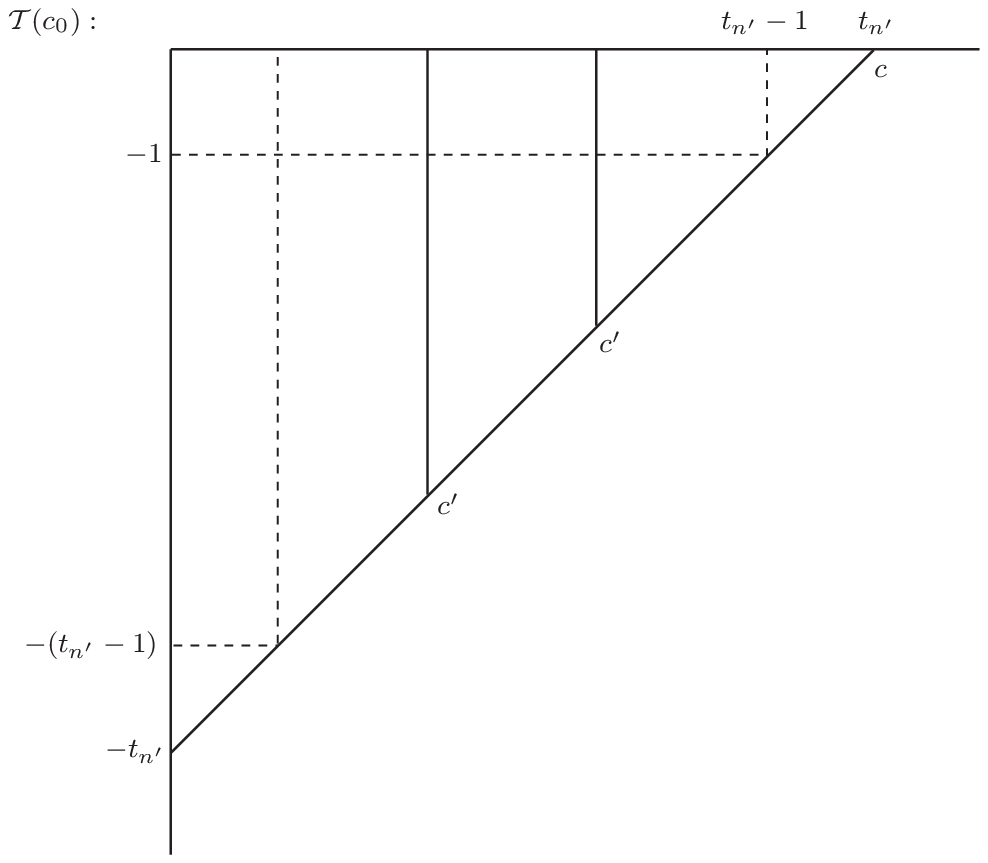}
\caption{} \label{case3}
\end{figure}

By the argument above,  one easily finds that
$\{P_{t_n}(c_0)\}_{n\ge 1}$ is a nested sequence of puzzle pieces
containing $c_0$ and $\deg(f^{t_n}: P_{t_n}(c_0)\to P_0(c))\le
\de^{b-1}$. Let $T_2$ be a topological disk compactly contained in
$P_0(c)$ such that $T_2$ contains all the puzzle pieces of depth $1$
in $P_0(c)$. Notice that all the critical values are contained in
the union of all the puzzle pieces of depth $1$ in $P_0(c)$, so they
are also contained in $T_2$.

Set $K_n(c_0):= P_{t_n}(c_0)$ and then we get similar objects as in
Lemma \ref{critn} (1). Using a similar method in the proof of Lemma
\ref{critn} (2), we could show that $H|_{\partial K_n(c_0)}$ admits
a qc extension inside $K_n(c_0)$ whose maximal dilatation is
independent of $n$.

(2) Since $c_0\in \mathrm{Crit_e}(f)$ and then $[c_0]=\{c_0\}$, we
know that $\cup_{c\in[c_0]}K_n(c)=K_n(c_0)$ and obviously $K_n(c_0)$
is nice.
\end{proof}

Summarizing Lemmas \ref{critp}, \ref{critn} and \ref{crite}, we have
proved Proposition \ref{equivalence}.

\appendix

\section{An application of Theorem \ref{rigidity}}

Cui and Peng proved the following result in \cite{cp} (see Theorem
1.1 in \cite{cp}).

\begin{theorem}
Let $U$ be a multiply-connected fixed (super)attracting Fatou
component of a rational map $f$. Then there exist a rational map $g$
and a completely invariant Fatou component $V$ of
$g$, such that\\
(1) $(f,U)$ and ($g,V$) are holomorphically conjugate, i.e., there
is a
conformal map mapping from $U$ onto $V$ and conjugating $f$ to $g$,\\
(2) each Julia component of $g$ consisting of more than one point is
a quasi-circle which bounds an eventually superattracting Fatou
component of $g$ containing
at most one postcritical point of $g$.\\
Moreover, $g$ is unique up to a holomorphic conjugation.

We say that $(g,V)$ is a holomorphic model for $(f,U)$.
\end{theorem}

To show the uniqueness of the model $(g,V)$, they divided the proof
into two parts. First they proved the following proposition
(Proposition 1.3 in \cite{cp}).

\begin{proposition}
Suppose that $g$ and $\tl g$ are two rational maps, and that $V$ and
$\tl V$ are two completely invariant Fatou components of $g$ and
$\tl g$ respectively satisfying the conditions (1) and (2) in
Theorem A.1. Then there is a qc map from the Riemann sphere $\cbar$
onto itself conjugating $g$ to $\tl g$ on $\cbar$ and this
conjugation is conformal on the Fatou set of $g$.
\end{proposition}
The other part is to show that the Julia set of the model $g$
carries no invariant line fields (Proposition 1.4 in \cite{cp}).

In this appendix, we will apply Theorem \ref{rigidity} (b) to give
an another proof of Proposition A.2.

\vspace{2mm}

\noindent\textit{Proof of Proposition A.2.} First starting from
$(g,V)$ and $(\tl g, \tl V)$, we construct $(g:\bf U\to \bf V)$ and
$(\tl g:\tl {\bf U}\to \tl {\bf V})$ satisfying the conditions as in
the set-up.

We may assume that the fixed points of $g$ and $\tl g$ in $V$ and
$\tl V$ are the point at infinity. By Koenigs Linearization Theorem
and B\"{o}ttcher Theorem, we can
take a disk $D_{0}=\{|z|>r\}\subset V$ such that\\
(1)$D_{0}\subset\subset g^{-1}(D_{0})$,\\
(2)$\partial D_{0}\bigcap (\cup_{n\ge 1}\cup_{c\in \mathrm{Crit}(g)}
\{g^n(c)\})=\emptyset$,\\
where $\mathrm{Crit}(g)$ denotes the critical point set of $g$.

Let $D_{n}$ be the connected component of $g^{-n}(D_{0})$ containing
$D_{0}$ for each $n\ge 1$. Then $D_{n}\subset\subset D_{n+1} $ and
$V=\bigcup_{n=0}^{\infty}D_{n}$. There is an integer $N_{0}$
satisfying that for any $n\geq 0$, $g^{-n}(D_{N_{0}})$ has only one
component, the set $\mathrm{Crit}(g)$ is contained in
$g^{-n}(D_{N_{0}})$  and every component of
$\cbar\smm\overline{D}_{N_{0}}$ contains at most one component of
$\cbar\smm V$ having critical points.

By Theorem A.1 (1), there is a conformal map $H:V\to \tl V$ with
$\tl g\circ H=H\circ g$ on $V$. Set ${\bf
V}:=\cbar\smm\overline{D}_{N_0}$, $\tl {\bf V}:=\cbar\smm
H(\overline{D}_{N_0})$ and ${\bf U}:= g^{-1}(\bf V)$ and $\tl {\bf
U}:= \tl g^{-1}(\tl {\bf V})$. One can check that $(g:\bf U\to \bf
V)$ and $(\tl g:\tl {\bf U}\to \tl {\bf V})$ satisfies the
conditions in the set-up.

Clearly, $\KKK_g=\cbar\smm V$ and $\KKK_{\tl g}=\cbar\smm \tl V$.
Since $H:V\to \tl V$ is a qc conjugacy from $g$ to $\tl g$ and ${\bf
V}:=\cbar\smm\overline{D}_{N_0}$, $\tl {\bf V}:=\cbar\smm
H(\overline{D}_{N_0})$, we know that $H:{\bf V}\smm \KKK_g\to \tl
{\bf V}\smm \KKK_{\tl g}$ is a qc conjugacy off $\KKK_g$. Let $E$ be
a periodic critical  component of $\KKK_g$ which is mapped to a
periodic critical component of $\KKK_g$ under some forward iterate
of $g$. According to Theorem A.1 (2), the component $E$ is a
quasi-circle which bounds an eventually superattracting Fatou
component containing a critical point $c$. In the proof of
Proposition 4.4 in \cite{cp}, a qc map $h_E$ is constructed. That
map is defined on a puzzle piece $P_{n_E}(c)$ containing $E$. From
the definition of that map, one can easily check that
$h_E|_{\partial P_{n_E}+n}(c)=H$ for all $n\ge 0$. Set $N_c:=n_E$
and let $M_c$ be the maximal dilatation of the map $h_E$. Then by
Theorem \ref{rigidity} (b), $H$ extends to a qc conjugacy off
$\mathrm{int}\KKK_g$. Notice that every component of
$\mathrm{int}\KKK_g$ is a bounded eventually superattracting Fatou
component. So $H$ extends to a conformal map in every component of
$\mathrm{int}\KKK_g$ which is again a conjugacy (refer to the proof
of Claim 4.1 in \cite{cp}). Hence $H$ extends to a qc conjugacy on
$\bf V$, that is, $H$ extends to a qc conjugacy on $\cbar$ which is
conformal on the Fatou set of $g$. \qed

\vspace{4mm}

 \noindent{\bf Acknowledgment.} The authors would like to thank Cui Guizhen, Jeremy
Kahn and Shen Weixiao for inspiring discussions.

{\footnotesize 

Addresses:

\indent {\footnotesize Wenjuan Peng, Institute of Mathematics,
Academy of Mathematics and Systems Science, Chinese Academy of
Sciences
 100190, P.R.China, wenjpeng@amss.ac.cn}\\

\indent {\footnotesize Lei Tan, D\'{e}partement de
Math\'{e}matiques, Universit\'e d'Angers, Angers, 49045, France,
Lei.Tan@univ-angers.fr}

\end{document}